\documentclass[12pt]{amsart}

\newif\ifIdiotsGuide

\newif\ifnotIdiotsGuide
\notIdiotsGuidetrue 

\ifIdiotsGuide
\notIdiotsGuidefalse
\fi

\usepackage{amssymb, stmaryrd}
\usepackage{hyperref}
\usepackage{csquotes}
\usepackage[capitalise]{cleveref}
\usepackage[margin=0.9 in]{geometry}
\usepackage{color,soul}
\usepackage{multicol}
\usepackage{multirow}
\usepackage[dvipsnames,svgnames,table]{xcolor}
\usepackage{enumerate}
\usepackage{mdframed}
\usepackage[T1]{fontenc}
\usepackage{tikz,soul}
\usetikzlibrary{positioning}
\usepackage{caption}
\usepackage{subcaption}
\usepackage{float}
\usepackage{comment}

\makeatletter
\def\bbordermatrix#1{\begingroup \m@th
  \@tempdima 4.75\p@
  \setbox\z@\vbox{%
    \def\cr{\crcr\noalign{\kern2\p@\global\let\cr\endline}}%
    \ialign{$##$\hfil\kern2\p@\kern\@tempdima&\thinspace\hfil$##$\hfil
      &&\quad\hfil$##$\hfil\crcr
      \omit\strut\hfil\crcr\noalign{\kern-\baselineskip}%
      #1\crcr\omit\strut\cr}}%
  \setbox\tw@\vbox{\unvcopy\z@\global\setbox\@ne\lastbox}%
  \setbox\tw@\hbox{\unhbox\@ne\unskip\global\setbox\@ne\lastbox}%
  \setbox\tw@\hbox{$\kern\wd\@ne\kern-\@tempdima\left[\kern-\wd\@ne
    \global\setbox\@ne\vbox{\box\@ne\kern2\p@}%
    \vcenter{\kern-\ht\@ne\unvbox\z@\kern-\baselineskip}\,\right]$}%
  \null\;\vbox{\kern\ht\@ne\box\tw@}\endgroup}
\makeatother

\def\VR{\kern-\arraycolsep\strut\vrule &\kern-\arraycolsep}
\def\vr{\kern-\arraycolsep & \kern-\arraycolsep}

\colorlet{ahi}{blue!15}

\theoremstyle{plain}
\newtheorem{theorem}[subsection]{Theorem}
\newtheorem{proposition}[subsection]{Proposition}
\newtheorem{lemma}[subsection]{Lemma}

\newtheorem{corollary}[subsection]{Corollary}
\newtheorem{theorem*}{Theorem}
\newtheorem{corollary*}{Corollary}

\theoremstyle{definition}
\newtheorem{definition}[subsection]{Definition}
\newtheorem{definition-proposition}[subsection]{Definition-Proposition}

\newtheorem{example}[subsection]{Example}

\newtheorem{notation}[subsection]{Notation}
\newtheorem{remark}[subsection]{Remark}

\newtheorem{caution}[subsection]{Caution}

\normalfont
\usepackage[T1]{fontenc}{}

\makeindex
\date{June 21, 2024}                                           
\title{Smooth Surfaces with Maximal Lines}
\author{Janet Page, Tim Ryan, and Karen E. Smith}
\email{janet.page@ndsu.edu, timothy.ryan.3@ndsu.edu, kesmith@umich.edu}
\thanks{The third author was partially supported by NSF-DMS grant numbers 2101075 and 2200501.}

\begin{document}

\maketitle

\ifIdiotsGuide
\textcolor{teal}{Throughout, we'll use teal for our ``Idiots Guide'' which can be removed later.  Comment out ``$\backslash$IdiotsGuidetrue'' in line 7 to get rid of these comments}
\fi

\begin{abstract} We prove that a smooth projective surface of degree $d$ in $\mathbb P^3$ contains at most $d^2(d^2-3d+3)$ lines. We characterize the surfaces containing exactly $d^2(d^2-3d+3)$ lines:
these occur only in prime characterize $p$ and, up to choice of projective coordinates,  are cut out  by equations of the form $x^{p^{e}+1}+y^{p^{e}+1}+z^{p^{e}+1}+ w^{p^{e}+1} = 0.
$
\end{abstract}

\setcounter{section}{0}
\section{Introduction}

In this paper, we prove a sharp  upper bound for the number of lines on a smooth algebraic surface  in $\mathbb P^3$ depending only on its degree and independent of the ground field, and then we precisely characterize the surfaces that achieve this bound.

Specifically, we prove the following theorem.

\begin{theorem}\label{main}
Let $S \subseteq\mathbb{P}^3$ be a smooth algebraic surface of degree $d > 3$ over an algebraically closed field $k$. 
Then $S$ contains at most $d^2(d^2-3d+3)$ lines. Furthermore, 
 $S$ contains exactly $d^2(d^2-3d+3)$ lines if and only if 
 \begin{enumerate}
 \item[(i)]
 $k$ has characteristic $p>0$;
 \item[(ii)] $d=p^e+1$ for some $e\in \mathbb N$; and 
 \item[(iii)] $S$ is projectively equivalent to the Fermat surface defined by 
 \begin{equation} \label{fermat}
x^{p^{e}+1}+y^{p^{e}+1}+z^{p^{e}+1}+ w^{p^{e}+1} = 0.
\end{equation}
\end{enumerate}
\end{theorem}

Previously, an upper bound of $d^4$ for the number of lines on surfaces in three space   was shown by Koll\'ar  \cite[Prop 53]{Kollar}. 
Little is known about {\it sharp} bounds beyond \cref{main}, although  considerable effort has gone into the case of quartics (\S \ref{quartics}).
Non-sharp quadratic bounds have been known for centuries when the field is restricted to characteristic zero (\S \ref{charzero}).

The special  Fermat surfaces starring in \cref{main} have  extremal properties in many other contexts. 
For example, the cones over them are "maximally singular" in the sense that they have the smallest possible {\it F-pure threshold}\footnote{The $F$-pure threshold is a singularity invariant in prime characteristic  that serves as a characteristic $p$ analog of log canonical threshold.} among all {reduced} cone singularities of the same multiplicity \cite[Thm  3.1, 4.3]{LowerBoundsExtremal}.
They are the only smooth surfaces (of degree $d\geq 3$)  in $\mathbb P^3$ with the property that all smooth plane sections are isomorphic \cite{Beauville.90} or 
with the property that 
 the dual surface is smooth \cite{Kleimann+Piene, Noma}. They also have the maximal number of $\mathbb{F}_{p^e}$-points among surfaces of the same degree \cite{Homma-Kim.13,Homma-Kim.15,Homma-Kim.16}. Notably, 
they have been consistently  used as examples of surfaces
with many lines in various settings, e.g., \cite{Bose-Chakravarti,Segre.65,Schutt+Shioda+vanLuijk,Degtyarev.15,Kollar}. 

Because it is helpful to name the surfaces  \eqref{fermat}, we recall
\begin{definition}\label{extremaldef}
Let $X =\mathbb{V}(f) \subseteq \mathbb{P}^n$ be a hypersurface over a field of characteristic $p$. We say that
$X$ is \textit{extremal} if $\mathrm{deg}(f) = p^e+1$ for some positive integer $e$ and  $f \in  (x_0^{p^e},x_1^{p^e},\dots, x_n^{p^e})$.  
\end{definition}
Extremal surfaces are classified in \cite{LowerBoundsExtremal}. 
In particular, a smooth surface $S$ in $\mathbb P^3$ is extremal if and only if $S$ is projectively equivalent to the Fermat extremal surface \eqref{fermat}; see \cite[Thm 6.1]{LowerBoundsExtremal}, \cite[Prop 2.3]{Shi01},  and
 \cite{Beauville.90} for three different proofs.\footnote{See also  \cite{Hasse-Witt}, \cite[Thm~1]{Lan56}, \cite[Thm~9.10]{Hef85}.} 
Using this terminology, the main result of the paper can be stated as follows: 
\begin{displayquote}
 \cref{main}: {\it A smooth surface  of degree $d$ in $\mathbb P^3$ contains at most $d^2(d^2-3d+3)$ lines, with equality 
 if and only if it is extremal.}
\end{displayquote}

Extremal hypersurfaces are  defined by $p^e$-bilinear forms analogously to how quadric surfaces are defined by  bilinear forms; for this reason, they inherit some features of quadrics, a fact  recognized by Shimada with his moniker {\it{$p$-quadrics}} \cite{Shi01}.  In particular, 
extremal surfaces are linked, via their automorphism groups, to the theory of algebraic groups over a finite field \cite{Lan56}. For example, the famous Drinfield curve of  Deligne-Lusztig theory is extremal \cite{Deligne+Lusztig}. The geometry of 
extremal surfaces was studied in the context of {\it Hermitian surfaces} over $\mathbb F_{p^{e}}$  by 
B. Segre and Bose-Chakravarti \cite{Segre.65, Bose-Chakravarti}; here the Frobenius map $r\mapsto r^{p^e}$ is an involution playing the role of complex conjugation in the classical theory. Their beautiful point-line geometry was developed further by Hirschfeld \cite{Hirschfeld}, and eventually in 
  the framework of finite incidence geometry by Payne and Thas; see \cite{PayneThasBook}.

In terms of the  lines on them, extremal hypersurfaces are more like {\it  cubics}  than quadrics---note that $3^2(3^2-3\cdot 3+3)$  is 27. The configuration of lines on  an extremal surface shares properties with the configuration of lines on  a cubic surface; for example,  they contain configurations analogous to Schl\"afli's double sixes \cite{Brosowsky+Page+Ryan+Smith}. In higher dimensions, Cheng---who calls extremal hypersurfaces "{\it $q$-bics}"---shows that the Fano variety of lines on a smooth extremal threefold  behaves similarly to that of cubic threefolds
\cite{Che23a}, \cite{Che23b}.

\smallskip
\subsection{Characteristic Zero}\label{charzero}
Of course,  bounding the maximal number of lines on  a smooth {\it complex} surface of degree $d$  is a classical subject with a history dating back at least to Clebsch's 1861 bound of $d(11d-24)$
\cite[p106]{Clebsch}. An improved bound due to Segre \cite{Segre43} held steady  as the best result for eighty years until Bauer and Rams recently proved  an upper bound of $11d^2-30d+18$;  their bound holds also in characteristic $p$,  provided that $p> d$ \cite{Bauer+Rams}.
Bauer and Rams's bound is not sharp, however, even when $d=4$.

\subsection{Lines on Quartics} \label{quartics}
Rams and Sch\"utt proved a
  sharp upper  bound of 64 lines on a quartic surface of characteristic zero or $p\geq 5$ \cite{Rams+Schutt.15-64lines}. 
Of course, in characteristic three,  \cref{main} guarantees that 112 is a sharp bound on the number of lines on a quartic surface, though this was first proved by 
Rams and Sch\"utt \cite{Rams+Schutt.15-112lines}.   In characteristic two,  $60$ is a sharp bound on the number of  lines on a quartic \cite{Rams+Schutt.18-AtMost64linesChar2, Degtyarev22}. 
Much effort has gone into understanding lines on quartic surfaces; 
 see also \cite{Degtyarev22,Degtyarev+Itenberg+Sertoz,Degtyarev+Rams,Gonzalez+Rams,Miyaoka,Rams+Schutt.14,Rams+Schutt.15-64lines,Rams+Schutt.15-112lines,Rams+Schutt.18-AtMost64linesChar2,Segre.43,Shimada+Shioda,Schur,Veniani-Char2,Veniani-K3Quartic,Veniani-QuarticSymmetries,Veniani-Char3}. See also \cite{Rams+Schutt.20-Quintics} for quintics.  

\medskip

Boissi\`ere and Sarti identify three ways to construct surfaces with many lines: those of the form $\mathbb{V}(\phi(x,y) + \psi(z,w))$, those branched over a curve in $\mathbb{P}^2$ with many total inflection points, and those with many automorphisms  \cite{Boissiere+Sarti}.  Perhaps unsurprisingly, extremal surfaces have all three of these properties. 
However, we do not understand well whether  there can be surfaces over a fixed field of characteristic $p$ with an abundance of lines---say, more than allowed by Bauer and Rams bound for small $d$---but not as many as on an extremal. One question that we do not know the answer to: {\it are there sequences of smooth surfaces in $\mathbb P^3$ admitting sets of lines whose cardinality grows cubically in $d$?}
The number of lines on extremal surfaces  grow quartically in the degree, but  for sequences of non-extremal surfaces,  the number of lines is bounded by a quadratic function in the degree in every class of examples of which we are aware.

\medskip

The proof of \cref{main} has three main steps, which make up the remaining three sections of the paper. 
We prove the bound  on the number of lines in  \cref{sec: line bound} using a result of Igusa to bound the Picard rank; extremal surfaces show immediately that the bound is sharp. 
\cref{uniqueness} begins the proof of the characterization of surfaces "with maximal lines"   by considering special configurations of lines on them and invoking a combinatorial result in incidence geometry due to Thas; specifically, we identify particular configurations of lines called {\it numerical star chord pair configurations} that we already know to exist on extremal surfaces.  
 \cref{last step} uses these configurations and some others to successively force strong restrictions on the defining equation of the surface in a particular choice of coordinates. The main idea, essentially, applies Bezout's theorem to understand when certain members in a net of lines in $\mathbb P^3$ lie on $S$.

	Throughout the paper, we work over an algebraically closed field $k$.

\section{An upper bound on the number of lines of a smooth surface}
\label{sec: line bound}

In this section, we prove the following upper bound on the number of lines on a smooth surface of degree $d$, together with  a restriction on the configuration of lines on a surface achieving this bound. This will establish the first part of \cref{main}. 

\begin{theorem}\label{planes around a line}
A smooth surface $S$ of degree $d$  in $\mathbb P^3$ contains at most $d^2(d^2-3d+3)$ lines.
Moreover, if $S$ contains {\it exactly} $d^2(d^2-3d+3)$ lines, then 
 {\it every}  line on $S$ lies in exactly $d^2-2d+2$ planes containing another line on $S$, and each such plane contains exactly $d$ distinct lines lying on the surface.
\end{theorem}

\begin{remark} Extremal surfaces of degree $d$ always contain exactly  $d^2(d^2-3d+3)$ lines---the maximal possible---so the bound of \cref{planes around a line} is sharp; see {\it e.g.}  \cite[Thm 3.3.1 (e)]{Brosowsky+Page+Ryan+Smith}. Ultimately, we will see that any surface "containing maximal lines" is extremal.
\end{remark}

The proof of \cref{planes around a line} proceeds as follows: we first use a result of Igusa to bound the Picard rank of $S$ (\cref{IgusaBound}), then use that bound to bound the number of lines intersecting a given line (\cref{max valence}). From there, we  use Bezout's theorem to conclude the result.

\subsection{Maximal Picard rank}

The N\'eron-Severi group of a  smooth projective variety is  the group of its divisors modulo algebraic equivalence. This is a finitely generated group whose rank is 
 called the Picard rank of the variety. Equivalently, the Picard rank can be defined as the rank of the group of divisors modulo numerical equivalence, since these two notions of equivalence agree on all divisors up to torsion \cite{Matsusaka}.   The Picard group of smooth surface $S\subseteq \mathbb P^3$ is torsion free in any case \cite[Thm 1.8]{SGAII} \ifIdiotsGuide \textcolor{teal}{ (technically the statement after it)}\fi,
so we freely interchange algebraic and numerical equivalence in this section.

\begin{proposition}\label{IgusaBound}
\label{picard rank} Let $S$ be a smooth surface in $\mathbb P^3$ of degree $d$. Then the Picard rank of $S$ is bounded above by
$d^3-4d^2+6d-2$.
\end{proposition}

\begin{remark}
Extremal Fermat surfaces  achieve this upper bound (as do all degree $d$ Fermat surfaces over fields of characteristic $p$ where $d$ divides $p^e+1$) \cite{Shioda+Katsura}, 
 so this bound is sharp.
\end{remark}

\begin{proof}[Proof of \cref{IgusaBound}]
 Igusa \cite{Igusa} shows that $\rho(S) \leq \chi(S) -2 +4\dim(\mathrm{Pic}^0(S))$   where  $\mathrm{Pic}^0(S)$ is 
the connected component of the identity in the Picard Scheme for $S$ and  $\chi(S)$ is the ``topological Euler characteristic,''
which in arbitrary characteristic can be defined as  \[\chi(S) = \int_{S} c(\mathcal{T}_{S}) = c_2(\mathcal{T}_S),\]
where $\mathcal{T}_S$ is the tangent bundle of $S$; see  \cite[Main Theorem]{Igusa} or \cite[Ex 18.3.7(c)]{Fulton}.

\ifIdiotsGuide
\textcolor{teal}{ Some of these references are vague on what field they are working over.
\cite{Igusa} gives the inequality above and says $\chi(S)$ is the second canonical class.
The discussion in \cite[Last remark]{Yamada} makes it clear that second canonical class means second Chern class of $\mathcal{T}_X$.
Alternatively, \cite[Ex 18.3.7(c)]{Fulton} shows how the degree of the top Chern class is an Euler characteristic, as an alternating sum of cohomology $\sum_{pq}(-1)^{p+q}H^p(X, \Omega_X^q).$ His introduction makes it clear he is working in arbitrary characteristic (and really over an arbitrary field). In general, you define the Euler characteristic using $\ell$-adic cohomology.
Other tangentially relevant references:
\begin{itemize}
\item  \cite{Kleimann+Piene} relates the topological Euler characteristic (in positive characteristic) to the top Chern class. 
\item \cite{Dolgachev} using the same $\ell$-adic Euler characteristic generalized a characteristic 0 formula for pencils (the 112 line authors seemingly use that formula - we use Picard rank to avoid it). 
\end{itemize}
}
\fi

 By \cite[7.2, Expose XI, Theorem 1.8]{SGAII}, $\dim({Pic}^0(S))=0$.
The desired bound follows since $c_2(\mathcal{T}_S)=d^3-4d^2+6d$, which can be seen by applying the additivity of the total Chern character to the normal bundle sequence of $S$ (see \cite[Examples 3.2.11-13]{Fulton}).
\ifIdiotsGuide \textcolor{teal}{ 
Indeed, using the multiplicitivity of total Chern classes on the exact sequence 
\[0 \longrightarrow \mathcal{T}_{S} \longrightarrow \mathcal{T}_{\mathbb{P}^3}\vert_S \longrightarrow \mathcal{N}_{\mathbb{P}^3/S} \longrightarrow 0,\]
where $\mathcal{N}_{\mathbb{P}^3/S}$ is the normal bundle of $S$ in $\mathbb{P}^3$, 
we see that 
\begin{equation} \label{cherncomp}
c\left(\mathcal{T}_{S}\right)  c\left(\mathcal{N}_{\mathbb{P}^3/S}\right) =  c\left(\mathcal{T}_{\mathbb{P}^3}\vert_S\right).\end{equation}
Since $\mathcal{N}_{\mathbb{P}^3/S} \cong \mathcal{O}_S(d)$, we have $c\left(\mathcal{N}_{\mathbb{P}^3/S}\right) = 1+dH|_S$. 
Using the Euler sequence for $\mathcal T_{\mathbb P^3}$, we compute that 
$c(\mathcal T_{\mathbb P^3}) = (1+H)^4 = 1+4H+6H^2+4H^3$; thus, restricting to $S$,  we see that
\[c\left(\mathcal{T}_{\mathbb{P}^3}\vert_S\right) = 1+4H|_S+6(H|_S)^2.\] 
Thus, \eqref{cherncomp} becomes
\begin{equation}\label{cherncomp2}
c\left(\mathcal{T}_{S}\right)  \left(1+dH|_{S}\right) = 1+4H|_S+6(H|_S)^2.
\end{equation}
This implies that
\[c\left(\mathcal{T}_{S}\right) = 1+(4-d)H|_S+(d^2-4d+6)(H|_S)^2,\]
so that 
$c_2\left(\mathcal{T}_{S}\right) = (d^2-4d+6)d$.  
Thus, $\rho(S) \leq  \chi(S) - 2 = d^3-4d^2+6d -2$ as desired.}
\fi
\end{proof}

\medskip
\subsection{Large Sets of Independent  Curve Classes}
As a step towards \cref{planes around a line}, we bound the number of lines on a surface $S$ intersecting a fixed line on $S$ in terms of the Picard rank. 
In order to state the result, we need some notation.
\begin{notation}\label{pencil notation}
Consider  a smooth surface $S\subseteq \mathbb P^3$ of degree $d\geq 3$, together with a fixed line $L$ on $S$.  Consider the set  of all lines on $S$ that intersect $L$ (excluding $L$ itself). Note that 
this is a finite set.
\ifIdiotsGuide
\textcolor{teal}{Consider a smooth surface $S$ of degree $d$. 
First, a useful fact. Any set of skew lines is numerically independent, so the Picard rank bounds the size of  skew set on $S$. \medskip\\
Fix a line $L$ on $S$. We claim that finitely many lines of $S$ intersect $L$. 
Consider the pencil of plane sections of $S$ containing $L$.
Any line of $S$ intersecting $L$ would lie in one of the fibers of this pencil. 
The lines in {\it different} fibers are skew by  \cref{threelines}. So we can have at most $\rho(S) $ fibers of the pencil containing lines. Since each fiber contains at most $d-1$ lines, the number of lines intersecting $L$ is bounded by $\rho(S)(d-1)$.\medskip\\
Now to see that  $S$ contains finitely many lines, fix a maximal skew set. Its cardinality is bounded by $\rho(S)$, and all remaining lines intersect at least one line in the skew set.
By the above argument, at most $\rho(S)\cdot\rho(S)\cdot(d-1)$ lines intersect that skew set, so there are at most $\rho(S)\cdot\rho(S)\cdot(d-1)+\rho(S)$ lines on $S$.
Thus, $S$ contains finitely many lines.}
\fi
Let 
$$
H_1, H_2, \dots,  H_s $$ 
be the complete list of planes in $\mathbb P^3$ that contain both $L$ and some other line of $S$. For each,  one of two possibilities holds for
the plane section $H_i \cap S.$   Let $\ell_i$ denote the number of lines on $S$ in the plane $H_i$ (other than $L$). Then 
either
\begin{itemize}
\item[(a).] $H_i \cap S =  L\cup L_{i,1}\cup\dots\cup L_{i, d-1} $ is a union of $d$ distinct  lines; {\sc or}
\item[(b).]  $H_i\cap S = L\cup L_{i,1}\cup\dots\cup L_{i,\ell_i} \cup C_i, $ where  $C_i$  is a degree $d-(\ell_i+1)\geq 2$ curve in $H_i$, none of whose irreducible components are lines. 
\end{itemize}
In Case (a), note that $\ell_i=d-1$, where as in case (b), $\ell_i\leq d-3$.
For each $i$, let 
\begin{equation}\label{**}
m_i= \begin{cases} 
d-2 \,\,\,\,{\text{in Case (a); }}  \\
\ell_i \,\,\, {\text{in Case (b).}}
\end{cases}
\end{equation}
\end{notation}

\begin{proposition}\label{independent classes}
Fix a line $L$ on a smooth surface $S$ of degree $d$. Let 
\begin{equation}
\mathcal L_i = \{L_{i,1}, L_{i,2}, \dots,  L_{i,m_i}\}
\end{equation}
be the set of coplanar lines on $S$ in the notation defined in \ref{pencil notation}. Let $h$ be an arbitrary plane section of $S$.
The set of curves 
\begin{equation}\label{independent set}
 \mathcal L_1 \cup \mathcal L_2 \cup \cdots \cup \mathcal L_s \cup \{L\} \cup \{h\}
\end{equation}
 represents a numerically independent set of curve classes on $S$.
\end{proposition}

Before proving \cref{independent classes}, for convenience, we 
record a simple fact that we use repeatedly: 

\begin{lemma}\label{threelines}
Three lines on a smooth surface in $\mathbb P^3$   intersect pairwise if and only if they are coplanar. 
\end{lemma}

\begin{proof} Coplanar lines always intersect, so we need only prove the converse. Let  $L_1$, $L_2$, and $L_3$ be three lines on a smooth surface $S\subseteq \mathbb{P}^3$. 
If their three pairwise intersection points are {\it distinct}, then these lines form a triangle obviously spanning a plane. If they all intersect at one point $p$, then for each $i$ we have 
$p\in L_i \subseteq T_pS$. Since $S$ is smooth, the tangent space $ T_pS$ is two-dimensional. Thus, the three lines are coplanar as they are contained in $T_pS$.
\end{proof}

\begin{proof}[Proof of \cref{independent classes}]
It suffices to show that the corresponding intersection matrix is invertible. To this end,  observe that 
\begin{equation}
h^2 = d \,\,\, \, {\text{and}} \,\,\,\, \,\, h\cdot M = 1
\end{equation}
for any line $M$ in \eqref{independent set}, including $L$.
Likewise, since $L\subseteq H_i$ for all $i$, and two lines in $\mathbb P^3$ intersect if and only if they are coplanar, 
 \begin{equation}
L\cdot L_{i,j}=1\,\, \,\,\, \, {\text{and}} \,\,\,\, \,\,  L_{i,j} \cdot L_{i',j'} = \delta_{ii'}
 \end{equation}
 for all indices $i, j , i', j'$.
Indeed, since $L_{i,j} \cap L_{i',j'}  \subseteq H_i\cap H_{i'}=L$, the lines $L_{i,j} $ and $ L_{i',j'}$ can only intersect at a point $p$ on $L$, in which case $L$, $L_{i,j} $ and $ L_{i',j'}$ pairwise intersect and so are coplanar by  \cref{threelines};  in this case $T_pS =H_i=H_{i'}$, so $i=i'$.

Finally, we claim that for any line $M$ on $S$, $M^2= 2-d$. Indeed, let $\Lambda$ be any plane containing $M$, and let $C$ be the corresponding residual curve, so that $\Lambda \cap S = M \cup C$. Then as classes, 
$
M \equiv h - C
$, 
and
\begin{equation} M^2 = h^2 - 2C\cdot h + C^2 = d - 2(d-1) +C^2= 2-d +C^2 = 2-d.
\end{equation}
To see that $C^2=0$,  compute $C\cdot C= C\cdot C'$  where $C'$ is the residual curve  for a different hyperplane section  $\Lambda' \cap S = M\cup C'$.
Note that $C\cap C' \subseteq \Lambda \cap \Lambda'=M$, so if $p$ is an  intersection point of $C$ and $C'$, then  $p\in M$.  So $p$ is a singular point of 
  both $C\cup M$ and $C'\cup M$. This would mean that both  planes $\Lambda$ and $\Lambda'$ are tangent to $S$ at $p$,  contrary  to the smoothness of $S$ at $p$.

  The previous calculations show that the 
 intersection matrix for the ordered set\footnote{Here, we order the set \eqref{independent set} so that the 
$L_{i,j}$ are ordered lexicographically,  followed by $L$, then $h$.}  \eqref{independent set} is 
  \equation \label{intersection}
\begin{bmatrix}
D            & \vec{1} & \vec{1}\\
\vec{1}^T &  2-d & 1 \\
\vec{1}^T &  1 & d \\
\end{bmatrix}
\endequation
where $D$ is a block diagonal matrix whose
diagonal blocks $A_1,\dots,A_s$ are the intersection matrices of the subsets  $\mathcal L_1,  \dots  \mathcal L_s, $ and $\vec 1$ is a column vector of all $1$'s  of the appropriate size.  Each $A_i$ block satisfies
\[A_i = \mathbf{1}_{m_i\times m_i}- (d-1)I_{m_i}.\] 
After a few invertible row and column operations, we can 
 convert the  matrix \eqref{intersection} to
\begin{equation}\label{new matrix}
\begin{bmatrix}
D            & \vec{0} & \vec{0}\\
\vec{0}^T &  1 & 0 \\
\vec{0}^T &  0 & 1 \\
\end{bmatrix},
\end{equation}
where $\vec 0$ is a column of $0$'s of the appropriate size.
To see that \eqref{new matrix} has non-zero determinant, we compute that  each block matrix $A_i$ has determinant  $(1-d)^{m_i-1}(m_i-(d-1))$, which is non-zero since  $0< m_i < d-1$.

Thus, the  intersection matrix \eqref{intersection} has non-zero determinant and the  set of numerical equivalence class of curves 
$\bigcup_{i=1}^s\{L_{i,j} \mid 1 \leq j \leq m_i\} \bigcup \{L, h\}$
 is  linearly independent, as desired.
  
\end{proof}

\begin{remark}
Alternatively, we can complete the proof of \cref{independent classes} 
without explicitly computing $\det(A_i)$ by applying  
Theorem 7.1.15 in \cite{Kollar+Lieblich+Olsson+Sawin} to the pencil of plane sections through $L$ on $S$. The set $\bigcup_{i=1}^s\{L_{i,j} \mid 1 \leq j \leq m_i\}$ is  a union of components of fibers that does not completely cover any fiber, so it must be 
algebraically (and hence numerically)  independent. This implies that the block diagonal matrix formed by $A_1, \dots, A_r$ is invertible, so we can proceed with the rest of the proof as above. 
\end{remark}

The following corollary is immediate.
\begin{corollary}
\label{pencil} Fix a line $L$ on a smooth surface $S$ of degree $d$. With notation as 
in  \cref{pencil notation}, we have
\[\rho(S) \geq 2+\sum_{i=1}^s m_i,\]
where $m_i$ is as defined in \eqref{**}.
\end{corollary}

\subsection{Bounding  the number of lines on smooth surfaces in $\mathbb P^3$.}
We  now use \cref{pencil} to establish a bound on the maximal number of lines that can intersect a given line on a smooth surface, and characterize when the maximal number is achieved. Extremal surfaces show that this bound is sharp.

\begin{proposition}\label{max valence}
Let $S\subseteq\mathbb{P}^3$ be a smooth surface of degree $d\geq 3$ containing a line $L$.
Then 
\begin{itemize}
\item[(i).] At most $d^3-3d^2+4d-2$ lines of $S$ intersect $L$ (not counting $L$ itself);
\item[(ii).] Exactly $d^3-3d^2+4d-2$  lines of $S$ intersect $L$   if and only if $L$  lies in a configuration of lines on $S$ formed by  $d^2-2d+2$ plane sections
 consisting of  $d$  lines of $S$, one of which is $L$.
 \end{itemize}
\end{proposition}

\begin{proof}
In the notation of \cref{pencil notation}, let 
$r$ be the number of planes $H_i$ satisfying (a)---that is, the number of planes $H_i$ containing $d$ lines of $S$, one of which is $L$. Note that $r$ can be zero. 
Remembering the definition of $m_i$ (see \eqref{**}), 
\cref{pencil} can be rewitten as 
\begin{equation}\label{rholowerbound}
    \rho(S) \geq  2+\sum_{i=1}^s \ell_i -r
\end{equation}
where $\ell_i$ is the number of lines on $S$ in $H_i$ other than $L$.
Putting \eqref{rholowerbound} and \cref{picard rank} together and adding $r-2$ to both sides, we have
\begin{equation}\label{2}
    d^3-4d^2+6d-4 + r \geq \sum_{i=1}^s \ell_i.
\end{equation}
By definition of $\ell_i$, the right-hand side of this equality is the number of lines on $S$ that intersect $L$.
Since $\ell_i = d-1$ for exactly $r$ indices $i$, we have that 
\begin{equation}
    (d-2)(d^2-2d+2)+ r =  d^3-4d^2+6d-4 + r  \geq r(d-1)
 \end{equation}
 which means that 
 \begin{equation}\label{rbound}
    d^2-2d+2 \geq r. 
 \end{equation}
Combining  \eqref{2} and \eqref{rbound}, we see that 
\[d^3-4d^2+6d-4+(d^2-2d+2) \geq  \sum_{i=1}^s \ell_i.\]
Thus, $d^3-3d^2+4d-2$ is an upper bound on the number of lines intersecting $L$, as desired.

For the final statement, first note that if $r<d^2-2d+2$, then 
\[\sum_{i=1}^s \ell_i \leq \rho(S)-2+r <  (d^3-4d^2+6d-2) -2+(d^2-2d+2) = d^3-3d^2+4d-2,\]
which is the maximal  possible number of lines intersecting a given line $L$ on $S$. 
Thus,  this maximum can be achieved only when  $r = d^2-2d+2$---that is only when there are exactly  $ d^2-2d+2$ planes each containing exactly $d$ lines of $S$, one of which is  $L$.
\end{proof}

We are now ready to prove \cref{planes around a line}.

\begin{proof}[Proof of  \cref{planes around a line}]
If $S$ does not have two skew lines, then it has at most $d$ lines since they are all contained in a single plane.
\ifIdiotsGuide
\textcolor{teal}{If there are no two skew lines, then every lines intersects every other line, and by \cref{threelines}, they all lie in a single plane.}
\fi
Thus, if $S$ has more than $d$ lines, we can choose coordinates so that $S$ contains the skew lines $L = \mathbb{V}(x,y)$ and $M = \mathbb{V}(z,w)$. In other words,  if $f$ is the homogeneous form defining $S$, then $f\in (x, y)\cap (z, w)$, so that
\begin{equation}\label{ff form}
f = xz f_1+xwf_2+yzf_3+ywf_4,
\end{equation}
where the $f_i$ are homogenous polynomials of degree $d-2$.

 Let us count the lines on $S$ that do not intersect $L$. 
An arbitrary line $N$  skew to $L$ in $\mathbb P^3$ is cut out by two homogenous forms, $h_1$ and $h_2$ such that the ideal $(x, y, h_1, h_2)$ is equal to $(x, y, z, w)$.
Thus, without loss of generality, we can assume $h_1$ and $h_2$ are the forms
\[
a_1x+a_2y - z \,\,\,\,\,\, {\text{and} } \,\,\,\,\,\,\,\, a_3x+a_4y - w,
\]
for some scalars $a_i$.
The line $N$ defined by these forms lies on $S$ if and only if  $f\in  (h_1, h_2)$, or equivalently, if and only if $f$ is in the kernel of the "restriction to $N$" map
\[
k[x, y, z, w ] \rightarrow k[x, y]\,\,\,\,\,\, z\mapsto a_1x+a_2y, \,\, \, w\mapsto a_3x+a_4y.
\]
Remembering that  the polynomial $f$ has the form \eqref{ff form}, we see that this map sends $f$ to a polynomial of degree $d$ in $x$ and $y$ whose coefficients 
are polynomials of degree $d-1$ in $a_1, a_2, a_3, a_4$.  In other words, $N$ lies on $S$ if and only if these degree $d-1$ polynomials in the four variables $a_1, a_2, a_3, a_4$ vanish. 
Since we know that $S$ has only finitely many lines, Bezout's theorem for affine space implies there are at most $(d-1)^4$ solutions to these polynomials in four variables. 
That is, there are at most $(d-1)^4$ lines on $S$ skew to $L$.

Invoking \cref{max valence}, we conclude that in total, there are at most 
\[(d-1)^4+(d^3-3d^2+4d-2)+1 = d^2(d^2-3d+3)\]
lines on $S$ as claimed.

Finally,  note that the   argument  above is sharp:  for any line  $L$, if there are {\it not } $d^2-2d+2$ planes containing $L$ and $d-1$  other lines of $S$,
then there are strictly fewer than $d^2(d^2-3d+3)$
lines on $S$.
\end{proof}

\begin{corollary}\label{maxbetweenskewlines}
	If $S$ is a smooth surface of degree $d>3$ which contains $d^2(d^2-3d+3)$ lines (the maximal possible according to  \cref{planes around a line}), then for any two skew lines  on $S$, there are $(d-1)^2 +1$ lines on $S$ intersecting both.
\end{corollary}
\begin{proof}
	Fix two skew lines $L$ and $M$ on $S$. By \cref{planes around a line}, $L$ lies in $(d-1)^2 +1$ planes containing other lines. Every line on $S$ intersecting $L$ must lie in one of those planes, and all lines in those planes intersect $L$. The line $M$ is skew to $L$, so in none of those planes. This means $M$ meets each of these planes. But  
 $M$ must intersect a single line in each plane by \cref{threelines}. The result follows. 
\end{proof}

\ifIdiotsGuide
\textcolor{teal}{Interestingly, we can show that: If $S$ is a smooth surface of degree $d>3$, then 
\begin{enumerate}
\item For any two skew lines  on $S$, there are \textbf{at most} $(d-1)^2 +1$ lines on $S$ intersecting both;
\item There are most $d$ through 3 skew lines (a degree 1 bound), 
\item There are at most $(d-1)^2+1$ through 2 skew lines (a degree 2 bound);
\item  There are  at most $((d-1)^2+1)(d-1)$ through 1 skew lines (a degree 3 bound);
\item  There are  at most $d^2(d^2-3d+3)$ through 0 skew lines (a degree 4 bound).
\end{enumerate}
}
\fi

\section{Numerical Star Chord Pair Configurations }\label{uniqueness}

 The configuration of lines on an extremal surface  of degree $d$ has several strong properties. 
In this section, we use a combinatorial result of  Thas to show that a smooth surface in 3-space with "maximal lines" admits a special configuration of lines that we call a {\it numerical star chord configuration} that always appear on extremal surfaces. In the next section, we will use this fact to show that surfaces with maximal lines {\it are}, in fact, extremal.
 
\begin{definition}\label{num star chord config} Let $S\subseteq \mathbb P^3$ be a surface of degree $d$. A {\bf numerical star chord configuration} on $S$ is 
a collection of  $d^2$ lines on $S$ that can be partitioned into $d$ sets of $d$ coplanar lines in {\it two different ways}.  \end{definition}

The purpose of this section is to prove the following crucial ingredient in the proof of \cref{main}, which will follow by combining our \cref{planes around a line} and Thas's \cref{thasCorollary}:

\begin{corollary}\label{has num star chord2} Let $S\subseteq \mathbb P^3$ be a smooth surface of degree $d\geq 3$.
If $S$ contains $d^2(d^2-3d+3)$ lines---the maximal possible by \cref{planes around a line}--- then $d-1$ is a prime power and $S$ contains a numerical star chord pair configuration.
\end{corollary}

Notice that \cref{has num star chord2} does {\it not} guarantee that $d-1$ is a power of the characteristic of the ground field. Indeed, \cref{cubic} shows that this need not be the case.
In
 \cref{last step},  we will show that when $d\geq 4$, under the hypothesis of \cref{has num star chord2},   $d-1$ is a power of the characteristic and the surface is extremal.

\medskip

The term ``numerical star chord pair configuration" is motivated by the following example. 

\begin{example}\label{star chord}\cite[Section 3]{Brosowsky+Page+Ryan+Smith}
Let $X\subseteq \mathbb P^3$ be an extremal surface of degree $d$.  Then whenever two lines on $X$ intersect at a point $p$, the plane they span (which is necessarily $T_pX$) intersects $X$ in a "star"---a union of $d$ coplanar lines through $p$. This plane is called a {\it star plane} and the point $p$ is called a {\it star point.}\footnote{When $d=3$, these are called Eckardt points. {Star points are also defined in greater generality in \cite{Cools-Coppens}.}}
If $p$ and $p'$ are two star points spanning a line $\ell$  {\it not} on $X$, then the line  $\ell$  is said to be a {\it star chord} of $X$. Every star chord intersects $X$ in 
 $d$ distinct points  $p_1, \dots p_d$, 
all of which turn out to be star points (two of which are $p$ and $p'$). The set 
\begin{equation}\label{partition}
\{ X \cap T_{p_1}X \} \cup \{ X \cap T_{p_2}X \} \cup \cdots \cup \{ X \cap T_{p_d}X \}
\end{equation}
consists of $d^2$ lines in $X$ partitioned into $d$ different sets of $d$ co-planar lines, namely the $d$ stars with star points $p_1, \dots p_d$.  
On the other hand, there is another way to partition these same $d^2$ lines into $d$ sets of $d$ coplanar lines. Namely, the star planes $T_pX$ and $T_{p'}X$ intersect in a star chord $\ell'$, which intersects $X$ in $d$ distinct star points $q_1, q_2, \dots, q_d$, none of which is any $p_i$.  The collection of $d^2$ lines
\begin{equation}\label{partition2}
\{ X \cap T_{q_1}X \} \cup \{ X \cap T_{q_2}X \} \cup \cdots \cup \{ X \cap T_{q_d}X \}
\end{equation}
is a {\it different} partition of the same $d^2$ lines \eqref{partition} into $d$ sets of $d$ coplanar lines.\footnote{As it turns out, $\displaystyle \cap_{i=1}^d T_{q_i}X=\ell$ and  $\displaystyle \cap_{i=1}^d T_{p_i}X=\ell'$
so the star chord pair $(\ell, \ell')$ is independent of the original choice of star points $p$ and $p'$; see \cite[Section 3]{Brosowsky+Page+Ryan+Smith}.}
We call these $d^2$ lines, together with the two partitions \eqref{partition} and \eqref{partition2},  a {\it star chord pair configuration;} it is a special case of a numerical star chord pair configuration.
\end{example}

\begin{remark}
Surfaces in $\mathbb P^3$  containing star chord configurations are a topic of considerable independent interest.  Any surface---extremal or not--- defined by $\phi(x, y)=\psi(z, w)$, where $\phi$ and $\psi$ are reduced forms of the same degree,  contains a star chord pair configuration  \cite{Caporaso+Harris+Mazur,Boissiere+Sarti,Xiao}. See also \cite[p116]{Hirschfeld}. \end{remark}

\begin{remark} There exist configurations of lines on smooth surfaces that are 
numerical  but not actual star chord pair configurations; see \cref{cubic} for examples on cubics.  Taking the pencil of planes between two surfaces of degree $d\geq 4$ each defined as the union of $d$ general planes, we see that the generic element of the pencil is an irreducible 
degree $d$ surface containing a numeric star chord configuration which is not a star chord pair configuration.

\ifIdiotsGuide
\textcolor{teal}{
Consider $2d$ general planes $\{H_1,\dots,H_d\}\cup \{K_1,\dots,K_d\}$ in $\mathbb{P}^3$, and let $L_{i,j} = H_i\cap K_j$.
If these lines lie on a surface of degree $d$, then they form a numerical star chord pair configuration which is not a star chord pair configuration.
If $H_i = \mathbb{V}(\ell_i)$ and $K_i = \mathbb{V}(m_i)$ then all of the lines lie on every degree $d$ surface in the pencil $\{s\prod_{i=1}^d \ell_i +t\prod_{i=1}^d m_i\,\, | \,\, [s:t]\in \mathbb P^1\}$. Are any of these smooth?}
\fi
\end{remark}

\begin{remark}\label{skew prop}
Suppose a smooth surface $S\subseteq \mathbb P^3$ of degree $d$ admits a numerical star chord configuration consisting of $d^2$ lines on $S$  partitioned by the  planes
$\{H_1, \dots, H_d\} $ and $\{K_1, \dots, K_d\}$.  In this case, the  $d^2$ lines on $S$ are precisely $L_{i,j} = H_i \cap K_j$ as $i$ and $j$ run through the set $\{1, 2, \dots, d\}$. 
Observe that
\begin{enumerate}
\item[(a)]
 The lines $H_i\cap H_j$ (respectively $K_i \cap K_j$) do not lie on $S$---that is, they are {\it chords};
 \item[(b)] The lines $L_{i,j}$ and $L_{m,n}$ are skew if $i\neq m$ and $j\neq n$.
\end{enumerate}
Indeed, (a) follows because otherwise the line $H_i\cap H_j$ would lie in both $H_i\cap S$ and $H_j\cap S$, contrary to the fact that $\{H_1, \dots, H_d\} $ {\it partitions} the lines.
\ifIdiotsGuide
\textcolor{teal}{Indeed, since $d$ planes, each of which can contain at most $d$ lines of $S$, partition exactly $d^2$ lines of $S$, they cannot share any lines of $S$.}
\fi
For (b), observe that   $L_{i,n}$ and $L_{m,n} $ are both in the plane $K_n$, and that $L_{i,j}$  intersects $L_{i,n}$.   
If $L_{i,j}$  intersects $L_{m,n}$  as well, then $\{L_{i,n}, L_{m,n}, L_{i,j}\}$ would pairwise intersect, so 
 $L_{i,j}$ would lie  in the plane $K_n$ by \cref{threelines}.
 Since $L_{i,j}$ is also in $K_j,$ this would imply that $L_{i,j} = K_n\cap K_j$, contrary to (a).
\end{remark}

\begin{example}\label{cubic}
Smooth cubic surfaces always contain 27 (which is $d^2(d^2-3d+3)$) lines, so by  \cref{has num star chord2}, they contain a numerical star chord pair configuration. 
Representing  a cubic surface  $S$ as the blow up of $\mathbb{P}^2$ at six points $p_1, \dots p_6$,  let $E_i$ denote the exceptional divisors, $C_i$ the strict transforms of conics through all $p_j$ with $j\neq i$, and $L_{ij}$  the strict transform of the line through $p_i$ and $p_j$.
Then one readily confirms that
\[\{E_1,C_2,L_{12}\}\cup \{E_3,C_4,L_{34}\}\cup\{L_{14},L_{23},L_{56}\}\]
\[\{E_1,C_4,L_{14}\}\cup \{E_3,C_2,L_{23}\}\cup\{L_{12},L_{34},L_{56}\}\]
is numerical star chord configuration on $S$. In fact, one can show that up to choosing the six (ordered) points to blow up, every numerical star chord pair configuration is of this type.

\ifIdiotsGuide
\textcolor{teal}{{\sc Proof of fact that all are of this type:}
Start with a numerical star chord pair configuration on a smooth cubic.
In particular, there are two skew lines in the configuration, we want to show first that up to labeling they are $E_1$ and $E_3$.\medskip\\
The intersection matrix of lines on any smooth cubic is the same as that on an extremal cubic.
Since the automorphism group of the extremal cubic is transitive on sets of three skew lines \cite{Brosowsky+Page+Ryan+Smith}, any pair of skew lines on a smooth cubic has the same intersection properties as $E_1$ and $E_3$.
In particular, every pair of skew lines can be extended to a set of six skew lines, which is then a part of a double six.
Therefore, up to labeling, any pair of skew lines is $E_1$ and $E_3$.
Thus, we can assume without loss of generality that $E_1$ and $E_3$ are in the configuration.\medskip \\
The only sets of three lines lying in a plane containing $E_i$ are of the form $\{E_i,C_j,L_{ij}\}$.
Picking the planes $\{E_1,C_2,L_{12}\}$ and $\{E_3,C_4,L_{34}\}$ is then just a labeling.\medskip\\
Once you have picked two planes, you have determined the numerical star cord so the claim follows.\medskip\\
}
\fi
\end{example}

\begin{caution}
Our definition of a {\it numerical} star chord configuration does {\it not} require that the sets of $d$ co-planar lines making up either partition form a star, as in \cref{star chord}. Choosing a sufficiently general cubic surface (one with no Eckardt points),  none of the coplanar sets of lines in \cref{cubic} form stars.
 \end{caution}

\subsection{Generalized quadrangles}\label{sec: num star chords} 
Generalized quadrangles were introduced by Tits in his study of simple groups of Lie type of rank 2 \cite{Tits}, and quickly became a popular research topic in combinatorics; see \cite{PayneThasBook} for a general reference. 
Our proof of \cref{has num star chord2} will make use of a characterization of certain generalized quadrangles due to Thas. We now briefly summarize the necessary machinery to state it.

An {\it incidence structure} is a triple $(P, B, I)$, where  $P$  and $B$ are sets (typically called  "points" and "lines" respectively\footnote{We will avoid the traditional use of "point'' and "line" as they would be confusing in our context.}) together with a relation 
$I\subseteq P \times B$ which determines when $p\in P$ is incident to $b\in B$. 
Two incidence structures  $(P,B,I)$ and $(P',B',I')$ are isomorphic if there are bijections $P\rightarrow P'$ and $B\rightarrow B'$ identifying $I$ with $I'$.

\begin{definition} \label{GenQuad} {\cite[p1]{PayneThasBook}}
A \textit{generalized quadrangle} with parameters $(s,t)$ is an incidence structure $(P,B,I)$ satisfying the following conditions: 
\begin{enumerate}
\item For all $x,x'\in P$, there is at most one $b\in B$ such that $(x,b),(x',b)\in I$.
\item For all $b,b'\in B$, there is at most one $x\in P$ such that $(x,b),(x,b')\in I$.
\item For all  $(x,b)\not\in I$, there exist unique  $x' \in P$ and $b' \in  B$ such that $(x',b), (x,b'),(x',b')\in I$.
\item For all $b\in B$, there are exactly $s + 1$ elements $x \in P$ with $(x,b)\in I$ . 
\item For all $x\in P$, there are exactly $t + 1$ elements $b \in B$ with $(x,b)\in I$ . 
\end{enumerate}
\end{definition}

Of particular interest in incidence geometry is the \textit{Hermitian generalized quadrangle}:

\begin{example}\label{Hermitian gq} \cite[3.1.1(ii)]{PayneThasBook}
Let $X=\mathbb V(x^{p^e+1} +y^{p^e+1} + z^{p^e+1} +  w^{p^e+1})\subseteq \mathbb P^3$ be the Hermitian surface (of degree $p^e+1$) defined over $\overline{\mathbb F_p}$.  
Let $P'$ be the set of $\mathbb F_{p^{2e}}$-points and let $B'$ be the set of lines on $X$ defined over $\mathbb F_{p^{2e}}$, with $I$ the natural incidence relation of points lying on lines.
Then $(P', B', I')$ is a generalized quadrangle with parameters $(p^{2e}, p^e)$ called the {\bf Hermitian  generalized quadrangle}.

 Equivalently, we can take $P'$ to be the  set of star points and $B'$ to be the set of  lines on {\it any}  smooth extremal surface $X$ degree $p^e+1$, since up to choice of projective 
 coordinates, $X$ can be assumed the Hermitian surface $\mathbb V(x^{p^e+1} +y^{p^e+1} + z^{p^e+1} +  w^{p^e+1})$ in which case the 
star points of a Hermitian surface are precisely its $\mathbb F_{p^{2e}}$-points;
see  \cite[Prop 2.5.1]{Brosowsky+Page+Ryan+Smith} or  \cite{Segre, Hirschfeld, Shi01}. \end{example} 

For our argument, we are interested in a {\it different}  generalized quadrangle associated to an extremal surface:

\begin{example}\label{Fermat gq}  Let $X$ be an extremal surface of degree $p^e+1$ in $\mathbb P^3$, let $P$ be the set of lines on $X$, and let $B$ be the set of star planes on $X$
(as defined in \cref{star chord} or \cite[Prop 2.5.1]{Brosowsky+Page+Ryan+Smith}). Defining incidence in the natural way, 
 $$I =\{(L, \Lambda)\,\, | \,\, L\subseteq \Lambda\}\subseteq P\times B, $$ 
 $(P, B, I)$ is a generalized quadrangle  with parameters $(s, t) = (p^e,p^{2e})$.
This is immediate from  \cite[Thm 3.3.1a]{Brosowsky+Page+Ryan+Smith}\ (but see also {\it e.g.} \cite{Hirschfeld} or \cite{Shi01}). In any case, it  is a special case of the {\it{a priori}} more general example given by the next result:
\end{example}

\begin{proposition}\label{gq}
Let $S\subseteq \mathbb P^3$ be a smooth surface of degree $d$ containing  $d^2(d^2-3d+3)$ lines.
Let $P$ be the set of lines on $S$, and let $B$ be the set of planes containing $d$ lines of $S$. With incidence defined the natural way, 
$(P,B,I)$ forms a generalized quadrangle with parameters $(d-1,(d-1)^2)$.
\end{proposition}

\begin{proof} This is  a corollary of  \cref{planes around a line}.
Conditions (1) and (2) hold as any two planes share at most one line and any two lines lie in at most one plane in $\mathbb P^3$.
Condition (3) holds because if a line $L$ on $S$ is not contained in a plane $H\in B$, then $L\cap H$ is a point on $S$, which must lie on a unique  line  $M \subseteq H\cap S$ (by \cref{threelines}); in this case, the plane spanned by $L$ and $M$ is in $B$ by  \cref{planes around a line}.  
Conditions (4) and (5) are the content of the second part of  \cref{planes around a line}.
\end{proof}

The \textit{dual} of an incidence structure $(P,B,I)$ is the incidence structure $(B,P,I^\vee)$ where $(b,p)\in I^\vee$ if and only if $(p,b)\in I$. The dual of a generalized quadrangle with parameters $(s, t)$ is a generalized quadrangle with parameters $(t, s)$.

\begin{proposition}\label{dualgq} Let $X\subseteq \mathbb P^3$ be a smooth extremal surface.  
Then the following  generalized quadrangles associated to $X$   are dual:
\begin{itemize}
\item[(i)] \cref{Fermat gq}, in which  $P$ consists of lines on $X$ and $B$ consists of star planes on $X$; 
\item[(ii)]
\cref{Hermitian gq}, in which  $P'$ consists of star points on $X$ and $B'$ consists of lines on $X$. 
\end{itemize} \end{proposition}

\begin{proof}
This follows immediately from the bijection between star points and star planes given by $p\leftrightarrow T_PX$ and the basic properties of the geometry of extremal surfaces proved, {\it e.g.}, in 
\cite{Brosowsky+Page+Ryan+Smith}. 
\end{proof}

A key step in the proof of  the main result of this section, \cref{has num star chord2}, will be to show that the generalized quadrangle of \cref{gq} is {\it also} dual to the Hermitian generalized quadrangle, and therefore isomorphic to the generalized quadrangle of lines and star planes on an extremal surface (\cref{Fermat gq}). To do so, we use a combinatorial characterization of the dual to the Hermitian generalized quadrangle due to Thas. To state this result, we need a few more definitions.

\begin{definition} \label{3-reg} Let $(P, B, I)$ be a generalized quadrangle with parameters $(s, t)$. 
\begin{itemize}
\item [(i).]
A \textit{triad} is a set of three elements $x_1, x_2, x_3 \in P$ such that if $(x_i,b) \in I$, then $(x_j,b)\not\in I$ for all $j\neq i$  \cite[p. 2]{PayneThasBook}.
\item [(ii).] 
A triad $\{x_1,x_2,x_3\}$ is \textit{3-regular} if it can be extended to a set of  $s+1$ elements $x_i\in P$ such that, whenever $y \in P$ is coincident with $x_1, x_2$ and $x_3$, then also $y$ is coincident with $x_4, \dots, x_{s+1}$; explicitly, this means that  if for  $i=1,2$ and $3$, there exists $b_i\in B$ such that 
 $(x_i,b_i),(y,b_i) \in I$, then also 
 there exists $b_i\in B$   for each $i= 4, 5, \dots, s+1$,  such that $(x_i,b_i),(y,b_i)\in I$   \cite[p. 4]{PayneThasBook}, \cite[Page 143]{Thas-94}. 
\end{itemize}
\end{definition}

\smallskip

\begin{theorem}\label{thasCorollary} \cite{Thas78} \cite{Thas-94} Let $(P, B, I)$ be a generalized quadrangle  with parameters $(q, q^{2})$  for some positive integer $q>1$.
If every triad of $(P, B, I)$ is 3-regular, 
then $(P,B,I)$  is isomorphic to the dual of the Hermitian generalized quadrangle. In particular, $q$ is a prime power.
\end{theorem}

\begin{proof} This is essentially  \cite[5.3.2.i]{PayneThasBook}, where the result is attributed to Thas in \cite{Thas78}. Note that the notation $Q(5, q)$ appearing in that statement is defined in 
 \cite[3.1.1]{PayneThasBook},  and shown to be isomorphic to the dual of the Hermitian quadrangle in  \cite[3.2.3]{PayneThasBook}, whenever $q$ is a prime power; that $q$ is a prime power is established in \cite[Corollary 3.1.4]{Thas-94}.
 \end{proof}

\subsection{Quadric Configurations}
In the generalized quadrangle $(P, B, I) $ of  a surface with maximal lines (defined in \cref{gq}), triads and 3-regularity have natural  geometric interpretations. For example, 
a \textit{triad} of elements in $P$ is a set of three skew lines on $S$.
A set of three skew lines is \textit{3-regular} if it can be extended to a set of  $d$ lines on $S$, each of which intersects every line on $S$ intersecting all three initial lines.   Our next step is to show that all triads in this generalized quadrangle are 3-regular. For this, we use the idea of a quadric configuration:

\begin{definition}\cite[Def 4.0.1]{Brosowsky+Page+Ryan+Smith} Let $S\subseteq \mathbb P^3$ be a smooth surface of degree $d\geq 3$. A {\it quadric configuration}  is a collection of $2d$ lines on $S$ partitioned into two sets of $d$ skew lines with the property that each line in either set meets every line of the other set.  

Equivalently, a quadric configuration is a set of $2d$ lines on $S$
whose union is the intersection  $S\cap Q$ for some smooth quadric surface $Q$ \cite[Prop 4.0.2]{Brosowsky+Page+Ryan+Smith}.  By necessity, $d$ lines of the quadric configuration  lie in each ruling of $Q$. 
\end{definition}

\begin{lemma}\label{quadrics} Let $S\subseteq \mathbb P^3$ be a smooth surface of degree $d\geq 3$. 
If $S$ contains $d^2(d^2-3d+3)$ lines, then 
every set of three skew lines on $S$ determines a quadric configuration on $S$. 
\end{lemma}

\begin{proof}
Any triple of skew lines in $\mathbb P^3$ determines a smooth quadric surface $Q$ \cite[2.12]{Harris}, with the three lines lying in one of the rulings of $Q$. Note that any line in $\mathbb P^3$ that intersects all three lines 
must lie on  $Q$, as its intersection multiplicity with the quadric $Q$ is greater than $2$. 
Since the intersection of $Q$ with our degree $d$ surface $S$ determines a curve of bi-degree $(d, d)$ on $Q$,   there are at most $d$ lines of $S$ (all in the other ruling of $Q$) which intersect the given three lines. That is,
 at most $d$ lines on $S$ can intersect every line in a set of three skew  lines in $\mathbb P^3$.

Fix a triple of skew lines  $\{L_1, L_2, L_3\}$ on $S$, and let $Q$ be the quadric surface they determine.
To prove \cref{quadrics}, we must show that the 
 bidegree $(d,d)$ curve $Q\cap S$  is a union of lines.

By \cref{max valence}, $L_3$ lies in $d^2-2d+2$ planes $H_i$, each of which contains $d$ lines of $S$. None of these $H_i$ contain $L_1$ or $L_2$ as both are skew to $L_3$.
Thus, $L_1$ and $L_2$ each intersect  $H_i$ in a single point, which necessarily lies on exactly one line in $H_i\cap S$ by \cref{threelines}.
If $L_1$ and $L_2$ intersect the same line in every $H_i \cap S$, then the three skew lines $\{L_1, L_2, L_3\}$ would intersect  $d^2-2d+2$ lines in common, which is impossible since at most $d$ lines on $S$ can intersect all three. 
Thus, there is some plane $H$ such that $L_3\subseteq H\cap S$ such that $L_1$ and $L_2$ intersect {\it different} lines in $H\cap S$.
Let  $N_1,  \dots, N_{d}$ be the lines of $H\cap S$ (where $N_3$ also denotes $L_3$ for convenience), and assume 
 without loss of generality that $L_1$ intersects $N_{1}$ and $L_2$ intersects $N_{2}$.  In particular, 
 $L_1$ and $L_2$ are skew to $N_i$ for $3\leq i \leq d$, and each set  $\{L_1, L_2,  N_i\}$ is a triple  of skew lines  on $S$.

By \cref{maxbetweenskewlines}, there are exactly $d^2-2d+2$ lines intersecting both $L_1$ and $L_2$; label them $M_i$.
No $M_i $ can lie in $ H$, since no line in $H\cap S$ intersects both $L_1$ and $L_2$. Thus, each $M_i$ intersects {\it exactly one} of the lines $N_j$ on $H$ for $1\leq j \leq d$ by \cref{threelines}.

We claim that {\it exactly one $M_i$ intersects $N_1$ and exactly one  $M_i$ intersects $N_2$.}
Indeed, consider the plane spanned by $L_1$ and $N_1$ (respectively $L_2$ and $N_2$). If $M_j$ and $M_k$ both intersect $N_1$ (respectively $N_2$), then both are in this plane (\cref{threelines}), so meet. In this case, $L_1$ and $ L_2$ 
are in this plane as well (again by  \cref{threelines}), contradicting the fact that $L_1$ and $L_2$ are skew.

Thus, there are  exactly  $d^2-2d+2-2=d(d-2)$ lines  intersecting $L_1$, $L_2$, and one of the $N_i$ with $3\leq i \leq d$.
Because $\{L_1, L_2, N_i\}$ is a triple of skew lines (for $3\leq i\leq d$), at most $d$ lines on $S$  can intersect all three.
So by the pigeonhole principle, for each $i=3, 4, \dots, d$,  there must be exactly $d$ lines intersecting $L_1$, $L_2$ and $N_i$. 
 \ifIdiotsGuide
\textcolor{teal}
{By the previous paragraph there are $d(d-2)$ lines that intersect $L_1$, $L_2$, and exactly one of $N_i$ for $3\leq i\leq d$. In particular, there are $d-2$ sets $\{L_1,L_2,N_i\}$ and those $d(d-2)$ lines each intersect all three lines in exactly one of those sets. Thinking of those $d-2$ sets as boxes and the $d(d-2)$ lines as pigeons, since each of the $d-2$ boxes can fit at most $d$ of the $d(d-2)$ pigeons, every box has exactly $d$ pigeons.
}
\fi
In particular, $L_1$, $L_2$, and $L_3$ intersect $d$ distinct lines of $S$ in common, say $M_1$, $\dots$, $M_d$, which all lie in the same ruling of the smooth quadric surface $Q$ determined by $L_1$, $L_2$, and $L_3=N_3$. Thus, $\{M_1, \dots, M_d\}$ is a set of $d$ skew lines in $S\cap Q$. 

Now, apply the same argument to the skew set  $\{M_1, M_2, M_3\}$ on $S$.  The same quadric $Q$ is determined, together with a set of $d$ skew lines  $\{L_1,  \dots, L_d\}$ lying in $S\cap Q$ in the other ruling. So   
$Q\cap S$ is a union the $2d$ lines $\{M_1, \dots, M_d, L_1, \dots, L_d\}$. \cref{quadrics} is proved.
\end{proof}

\begin{theorem}\label{has num star chord1} Let $S\subseteq \mathbb P^3$ be a smooth surface of degree $d\geq 3$ containing $d^2(d^2-3d+3)$ lines.
Then $d = p^e+1$ for some prime $p$ and positive integer $e$, and the generalized quadrangle of lines on $S$ (defined in \cref{gq}) is isomorphic to the generalized quadrangle of lines of an extremal surface (\cref{Fermat gq}).
\end{theorem}

\begin{proof} Let $(P, B, I)$ be the generalized quadrangle where $P$ is the set of lines on $S$ and $B$ is the set of planes containing $d$ lines of $S$. The parameters of this  generalized
quadrangle are $(d-1, (d-1)^2)$.
An arbitrary  triad in this generalized quadrangle is a triple  $\{L_1, L_2, L_3\}$ of skew lines on $S$. By  \cref{quadrics}, any such triad  determines a quadric configuration 
$$
S\cap Q =\{L_1, L_2, L_3, \dots, L_d\} \, \bigcup \, \{M_1, M_2, \dots, M_d\}  
$$ on $S$ where the $L_i$'s are in one of the rulings and the $M_i$'s the other.
So we can extend the triad $\{L_1, L_2, L_3\}$ to a set 
$\{L_1, L_2, L_3, \dots L_d\} $ of $d$ skew lines on $S$.
These have the property that for any line $M$ on $S$ intersecting  $L_1, L_2$ and $L_3$,  $M$ must also intersect $\{L_4, \dots L_d\}.$  Since two lines on $S$ intersect if and only if they are coplanar, 
 this exactly says that the  triad $\{L_1, L_2, L_3\}$ of the generalized quadrangle  $(P, B, I)$ is $3$-regular in the sense of \cref{3-reg}.  Thus, the hypothesis of \cref{thasCorollary}  is satisfied, and we conclude that  $d-1$ is  a prime power $p^e$ and that the incidence structure $(P, B, I)$ is isomorphic to the dual of the Hermitian generalized quadrangle using \cref{thasCorollary}. By \cref{dualgq}, $(P, B, I)$ is isomorphic to the  generalized quadrangle of lines and star planes on  an extremal surface, as described in \cref{Fermat gq}. 
\end{proof}

\begin{proof}[Proof of \cref{has num star chord2}]  \cref{has num star chord1} shows that $d-1$ is a prime power. 
Since two lines on $S$ intersect if and only if they lie in a plane together (which then must have $d$ lines of $S$ by \cref{planes around a line}), \cref{has num star chord1} shows that  the 
intersection matrix of the lines  on any surface with the maximal number of lines is identical to the intersection matrix of the lines on an extremal surface (for appropriate ordering of the lines).
The existence of a numerical star chord pair configuration on a surface depends only on the intersection matrix of its lines. So
since extremal surfaces have (numerical) star chord pair configurations \cite[Thm 3.1.4]{Brosowsky+Page+Ryan+Smith}, so does $S$. 
\end{proof}

\section{Surfaces with Maximal Lines are Extremal}\label{last step}

In this section, we complete the proof of \cref{main} by proving the following:

\begin{theorem}\label{maintheoremsec3} Suppose $S \subseteq \mathbb P^3$ is a smooth surface of degree $d>3$ containing the maximal number of lines, $d^2(d^2-3d+3)$.  Then $S$ is extremal. That is, the characteristic of the ground field is $p>0$, $d=p^e
+1$ for some $e\in \mathbb N$, and $S$ is projectively equivalent\footnote{over the algebraic closure of the ground field} to the Fermat surface defined by $x^{p^e+1}+y^{p^e+1}+z^{p^e+1}+w^{p^e+1}.$
\end{theorem}

Our strategy is to repeatedly use the numerical star chord configuration on $S$ guaranteed by \cref{has num star chord2},  together with the configurations of lines on $S$ intersecting a pair of skew lines  guaranteed by \cref{maxbetweenskewlines}, to get successively stronger restrictions on the monomials that can appear in a defining equation $f$ for $S$ in an appropriate choice of coordinates. Interestingly, although \cref{has num star chord2} guarantees that the degree $d$ of $f$ must satisfy $d=p^{e}+1$ for {\it some} prime $p$, it does {\it not} guarantee that $p$ is the characteristic of the ground field. This will  require proof.

\begin{lemma}
\label{num star chord equation} Let $S\subseteq \mathbb P^3$ be a smooth surface of degree $d$ containing a 
 numerical star chord configuration.  Then without loss of generality the defining equation of $S$ can be taken to be 
\begin{equation}\label{f form}
f = xw \prod_{i=1}^{d-2} \ell_i + yz \prod_{i=1}^{d-2}m_i
\end{equation}
 where $\ell_i$ and $m_i$ are distinct (up to scalar multiple)  linear 
polynomials in $x,y,z,w$.
\end{lemma}

\begin{proof}
Let  $\{H_i \,\,| \,\, i=1, \dots, d\}$ and $\{K_j \,\,| \,\, j=1, \dots, d \}$ be the two collections of $d$ planes each containing $d$ of the  $d^2$ lines forming a numerical star chord pair configuration on $S$. Suppose that $H_i $ and $K_j$ are defined by linear forms $\ell_i$ and $m_j$, respectively. The $d^2$ distinct lines of the configuration are therefore $\mathbb{V}(\ell_i,m_j) = H_i\cap K_j$.

Let $f$ be the degree $d$ defining equation of $S$.
Since $H_i\cap S$ is a union of the $d$ lines $\mathbb{V}(\ell_i,m_1),  \mathbb{V}(\ell_i,m_2), \dots, \mathbb{V}(\ell_i,m_d),$ we have
\[
S\cap H_i = \mathbb{V}(\ell_i, f)  =   \mathbb{V}(\ell_i, \prod_{j=1}^d m_j).
\]
So  $f$ is in the ideal $ (\ell_i, \prod_{j=1}^d m_j),$
as this ideal   is  radical in $k[x, y, z, w]$.
Thus, for each $i$,  we can write 
\[
f = \ell_i g_i + \lambda_i \prod_{j=1}^d m_j
\]
for some form $g_i$ of degree $d-1$ and some non-zero constant $\lambda_i$. Replacing each linear form $m_j$ by a suitable scalar multiple, we can assume that for each $i$, there is a degree $d-1$-form $g_i$ such that 
\[
f - \prod_{j=1}^d m_j =  \ell_i g_i.
\]
But now the polynomial $f-\prod_{j=1}^d m_j$ is a degree $d$ form in $k[x, y, z, w]$ divisible by {\it each} of the linear  forms $\ell_1, \ell_2, \dots, \ell_d$, so  it must be 
$\prod_{i=1}^d \ell_i $ (up to non-zero scalar multiple) by  the UFD property of 
the polynomial ring.  Without loss of generality, $f= \prod_{i=1}^d \ell_i+\prod_{j=1}^d m_j$.

Finally, since the lines $\mathbb V(\ell_1, m_1)$ and $\mathbb V(\ell_2, m_2)$ are skew in $\mathbb P^3$ (\cref{skew prop}),
the ideal generated by $\ell_1, \ell_2, m_1,  m_2$ must be 
$(w, x, y, z)$, so we can choose coordinates to make $\ell_1,  \ell_2, m_1, m_2$ equal to $w, x, y, z$, respectively. This ensures $f$ has the form \eqref{f form}.
\end{proof}

\begin{notation}\label{standardf}
Throughout the rest of this section, we will assume $S = \mathbb{V}(f)\subseteq \mathbb P^3$ is a smooth surface of degree 
$d >3$ which contains a numerical star chord configuration.  By \cref{num star chord equation}, we write $f$ in the following way
	\begin{equation}\label{num star chord generic eq}
		f = xw \prod_{i=1}^{d-2} (a_i x+b_iy+c_iz+d_i w)+ yz \prod_{i=1}^{d-2} (e_ix +f_iy+g_i z+h_iw).
	\end{equation}
\begin{lemma}\label{zerofgad} 
Suppose that a smooth surface in $\mathbb P^3$ has defining equation of the form \eqref{num star chord generic eq}. 
Then none of the coefficients  $a_i$,  $d_i$,  $f_i$ or  $g_i$ is zero for any $i=1, \dots, d$.
\end{lemma}

\begin{proof} By symmetry, it suffices to show that no $a_i$ is zero.
Suppose that $a_i=0$. Then for $\ell_i=b_iy+c_iz+d_i w$, 
the point $[1:0:0:0]$ lies on both lines
$\mathbb V(w, y)$ and $\mathbb V(\ell_i, z)$. In particular, this pair of lines is not skew, contrary to \cref{skew prop}(b).
\end{proof}
\end{notation}

\begin{proof}[Proof \cref{maintheoremsec3}]
By \cref{has num star chord2}, $S$ has a numerical star chord pair configuration.
By \cref{num star chord equation} and \cref{zerofgad},  we may assume that the defining equation of $S$ has the form
\[
f = xw \prod_{i=1}^{d-2} (a_ix+b_iy+c_iz+d_i w)+ yz \prod_{i=1}^{d-2} (e_ix +f_iy+g_i z+h_iw)
\]
where no $a_i, d_i, f_i$ or $g_i$ is zero.

By \cref{maxbetweenskewlines}, there should be exactly $(d-1)^2+1$ lines on $S$ intersecting both of the  skew lines $\mathbb V(x,z)$ and $\mathbb V(y,w)$ of $S$. One of these lines is 
$\mathbb V(z,w)$. The others are all defined by
\begin{equation}\label{line form}
x=Az\,\,\,\,\,\, \text{ and }  \,\,\,\,\,\, y = Bw 
\end{equation}
for some scalars $A, B\in k$.
Restricting the polynomial $f$ to  an arbitrary line $L$ given by equations of the form \eqref{line form}, we get a polynomial $\overline f$ in two variables which is zero if and only if $L$ lies on $S$. Algebraically, we can write $\bar f$ as the image of $f$ under the map 
\[
k[x, y, z, w] \rightarrow k[z, w] \,\,\,\,\,\, \,\,\,\,\,\ x\mapsto Az,\,\,\,y \mapsto Bw
\]
whose kernel is the ideal of $L$.
Explicitly, in these coordinates
\begin{equation}\label{f-bar}
\bar{f} = Azw \prod_{i=1}^{d-2} ((Aa_i+c_i)z+(b_iB+d_i)w)+ Bzw \prod_{i=1}^{d-2} ((e_iA+g_i)z +(Bf_i+h_i)w).
\end{equation}
Since there are exactly $(d-1)^2$ lines of type \eqref{line form} on $S$, there  are exactly $(d-1)^2$ choices for $(A,B)$ where $\bar f\in k[z, w]$ is the zero polynomial.  

To analyze this, let  $F_\ell \in k[A, B]$ be the coefficient of $z^\ell w^{d-\ell}$ in $\bar{f}$  as expressed in \eqref{f-bar}.  The line $L$ given by the equations \eqref{line form} lies on $S$ if and only if 
\begin{equation}\label{F's}
F_1(A, B)=F_2(A, B)=\cdots = F_{d-1}(A, B)=0.
\end{equation}
 In particular, there are exactly $(d-1)^2$ solutions to the equations \eqref{F's}. Looking at just $F_1$ and $F_{d-1}$, we see that 
\[F_{1} = A \prod_{i=1}^{d-2} (b_iB+d_i)+ B \prod_{i=1}^{d-2} (f_iB + h_i) 
\,\,\,\, \text{ and } \,\,\,\,
F_{d-1} = A \prod_{i=1}^{d-2} (a_iA+c_i)+ B \prod_{i=1}^{d-2} (e_iA + g_i).\]
As neither $\prod a_i$ nor $\prod f_i$ is zero (\cref{zerofgad}), the polynomials  $F_{d-1}$ and $F_1$ are both degree $d-1$. Furthermore, $F_1$ and $F_{d-1}$ have no non-trivial common factor: 
 viewing $F_1$ as a degree one polynomial in the variable $A$ over the field $k(B)$, we see that any factor would be a polynomial in $B$ alone, and likewise any factor of $F_{d-1}$
 would be a polynomial in $A$ alone; hence any common factor is a scalar. Hence
Bezout's theorem tells us that $F_1$ and $F_{d-1}$ can have {\it at most}  $(d-1)^2$ common  zeros. On the hand,  imposing the further conditions that 
{\it all } the  $F_\ell$ in \eqref{F's} vanish, we have exactly 
 $(d-1)^2$ distinct zeros.  Thus, the degree $d-1$ polynomials $F_1$ and $F_{d-1}$ together must cut out the full set of all $(d-1)^2$ solutions to 
\eqref{F's}. In particular, the ideal $(F_1, F_{d-1})$ must be the radical ideal describing all $(d-1)^2$ solutions to 
\eqref{F's}, and every other $F_\ell$ must be in the ideal $(F_1, F_{d-1})$.

 \ifIdiotsGuide
\textcolor{teal}{{\sc Another way to see no common factors:}
A priori, $F_1$ and $F_{d-1}$ could share a factor, if they do not then the statement is immediate.
Presume the product of the common factors of $F_1$ and $F_{d-1}$ has degree $e$, then their intersection is a degree $e$ curve $C$ and a finite subscheme of length $(d-1-e)^2$. 
Let $C_1$, $\dots$, $C_m$ be the irreducible components of $C$ of degrees $e_1$, $\dots$, $e_\ell$, respectively.
Since there are finitely many solutions, there must be some $F_{\ell_i}$, which has degree at most $d-1$, that does not contain $C_i$.
Thus, there are at most $e_i\cdot (d-1)$ solutions on $C_i$.
Adding these solutions up, there is at most $(d-1)e$ solutions on $C$. 
Thus, there are at most $(d-1-e)^2+(d-1)e = (d-1)^2-e(d-1)+e^2$ solutions.
Since there must be $(d-1)^2$ solutions, $(d-1)^2-e(d-1)+e^2\geq (d-1)^2$ or in other words, $e^2\geq (d-1)e$.
Since $F_1$ and $F_{d-1}$ are not scalar multiples as they have distinct sets of monomials, $e<d-1$ so $e=0$.
Thus, the intersection of $F_1$ and $F_{d-1}$ is $(d-1)^2$ (distinct) points, so the set of zeros of $(F_1,\dots,F_{d-1})$ is their complete intersection.
} 
\fi

The  polynomial $F_\ell \in k[A,B]$  is 
\begin{equation}\label{Fl}
F_{\ell} = \sum_{\sigma \subseteq\{1,\dots,d-2\}, |\sigma| = \ell-1} A\prod_{i\in \sigma}(a_iA+c_i) \prod_{i\notin \sigma}(b_iB+d_i )+ B \prod_{i\in \sigma} (e_iA + g_i) \prod_{i\notin \sigma}(f_iB + h_i).
\end{equation}
When $1<\ell<d-1$, using the fact that $F_\ell$ does not have an $A^{d-1}$ or a $B^{d-1}$ term,  it can be shown that $F_\ell$ is the trivial linear combination of $F_1$ and $F_{d-1}$---that is, $F_\ell=0$. (For example, homogenize
$F_\ell, F_1$ and $F_{d-1}$ to  degree $d-1$ polynomials in $k[A, B, C]$ and observe that $F_1$ and $F_{d-1}$ cut out $(d-1)^2$ points in the affine chart of $\mathbb P^2$ where $C\neq 0$.)
 \ifIdiotsGuide
\textcolor{teal}{{\sc In detail: } We can homogenize  $F_1, \dots, F_{d-1}$ to get degree $d-1$ polynomials $\tilde F_1, \dots, \tilde F_{d-1}$ in $k[A, B, C]$. Since $(F_1, F_{d-1})$ cuts out a reduced scheme of $(d-1)^2$ points in $\mathbb A^2$,  the degree $d-1$ polynomials $\tilde F_1, \tilde F_{d-1}$  cut out $(d-1)^2$  distinct points in $\mathbb P^2$.
  If $F_\ell\in (F_1, F_{d-1})$, then for some $n$,
 $$
 C^n\tilde F_\ell\in (\tilde F_1, \tilde F_{d-1}) = \bigcap_{j=1}^{(d-1)^2} \mathfrak P_j,
 $$
where the $\mathfrak P_j$ are the homogeneous prime ideals of the $(d-1)^2$ points. Since those points all lie in the affine chart $\mathbb P^2$ where $C\neq 0$, 
no $\mathfrak P_j$ contains $C$. Thus
 $
 \tilde F_\ell\in (\tilde F_1, \tilde F_{d-1}).
 $
Since $\tilde F_1,\tilde F_{d-1}$ and $ \tilde F_\ell$  are {\it homogeneous} elements of the same degree, $\tilde F_{\ell}$ must be a {\it scalar} combination of $\tilde F_1, \tilde F_{d-1}$. Specializing to $C=1$, also
 $F_\ell$ must be a {\it scalar} combination of $F_1$ and $F_{d-1}$. But now for $1<\ell <d-1$, $F_\ell$ does not have an $A^{d-1}$ or a $B^{d-1}$ term,  so it must be the trivial linear combination of $F_1$ and $F_{d-1}$, i.e., $F_\ell$ must be the zero polynomial.
} 
\fi

We now relate the coefficients of the monomials in $A$ and $B$ in $F_\ell$ to the coefficients of the monomials of $x,y,z,w$ in $f$.
Write $f$ as 
\[f = \sum_{j+k+m+n=d} a_{jkmn}x^jy^kz^mw^n.\]
In this notation, we have 
\[\bar{f} = \sum_{j+k+m+n=d} a_{jkmn}(Az)^j(Bw)^kz^mw^n = \sum_{j+k+m+n=d} a_{jkmn}A^jB^kz^{j+m}w^{k+n}.\]
As we have defined them, $F_\ell$ is the coefficient of $z^\ell w^{d-\ell}$ in $\bar{f}$.
In other words, 
\[F_\ell = \sum_{j+m=\ell,k+n=d-\ell}  a_{jkmn}A^jB^k\in k[A, B].\]
Thus,  the coefficient of $A^jB^k$ in the polynomial $F_\ell$ is $a_{jkmn}$, for $\ell = j+m$ .
\ifIdiotsGuide 
\textcolor{teal}{
(Note that although $n$ does not explicitly appear, it is computed as $d-\ell-k$.)
}
\fi
Since $F_\ell$ is the zero polynomial in $A$ and $B$ for $1<\ell<d-1$, we see that   $ a_{jkmn} = 0$ unless $j+m=1$ or $j+m=d-1$.
In other words, $f$ only has monomials of the form 
\begin{equation}\label{set1}
\{xy^kw^{d-k-1},y^kzw^{d-k-1},x^{j}yz^{d-j-1},x^{j}z^{d-j-1}w\}.
\end{equation}

Symmetrically, if we interchange the roles of $y$ and $z$ in the preceding argument, we see that 
$f$ only has monomials of the form 
\begin{equation}\label{set2}
\{xz^mw^{d-m-1},yz^mw^{d-m-1},x^{j}y^{d-j-1}z,x^{j}y^{d-j-1}w\}.
\end{equation}

Comparing \cref{set1} and \cref{set2}, we see that $f$ only contains monomials of the 
\[\{xw^{d-1},x^{d-1}w, xy^{d-2}w,xz^{d-2}w, yz^{d-1},y^{d-1}z, x^{d-2}yz,yzw^{d-2}\}.\]
Thus, we can write  
\begin{equation}\label{12}
f = xw(A_1x^{d-2}+A_2y^{d-2}+A_3z^{d-2}+A_4w^{d-2})+yz (B_1x^{d-2}+B_2y^{d-2}+B_3z^{d-2}+B_4w^{d-2})
\end{equation}
for some scalars $A_i, B_i$.
Since we can also write 
\begin{equation}\label{13}
f=xw\prod_{i=1}^{d-2}(a_ix+b_iy+c_iz+d_iw)+yz \prod_{i=1}^{d-2}(e_ix+f_iy+g_iz+h_iw),
\end{equation}
and no term on the left hand side of \eqref{12} is divisible by $yz$, we can conclude that
\begin{equation}\label{poly}
\prod_{i=1}^{d-2}(a_ix+b_iy+c_iz+d_iw) = A_1x^{d-2}+A_2y^{d-2}+A_3z^{d-2}+A_4w^{d-2}.
\end{equation}

The left hand side of \eqref{poly} is a product of {\it distinct} linear factors, so by the UFD property of the polynomial ring, the same is true for the right hand side. 
Because $A_1 =\prod_{i=1}^{d-2} a_i \neq 0$ and $A_4=\prod_{i=1}^{d-2} d_i\neq 0$ (\cref{zerofgad}), we claim this forces  $A_2=A_3=0$. Indeed otherwise 
 $A_1x^{d-2}+A_2y^{d-2}+A_3z^{d-2}+A_4w^{d-2}$ is a sum of at least three non-zero diagonal terms, which is either irreducible or a power of an irreducible polynomial (the latter only in the case that 
$d-2$ is divisible by the characteristic of the ground field),  rather than a product of distinct linear factors.
Symmetrically, because $B_2, B_3\neq 0$, we know  $B_1=B_4=0$.
Finally, scaling the coordinates appropriately, we can assume that 
\begin{equation}\label{almost}
f  = x^{d-1}w+xw^{d-1}+y^{d-1}z+yz^{d-1}.\
\end{equation}
\ifIdiotsGuide
{\color{teal}{
Indeed, scaling $x$ by $x \mapsto \alpha x$ where $\alpha^{d-2} = \frac{A_4}{A_1}$ we get
 \[f  = \alpha A_4(x^{d-1}w+xw^{d-1}) + B_2y^{d-1}z+B_3yz^{d-1}.\]
A similar scaling on $y$ sends our polynomial to 
\[
f  = \alpha A_4(x^{d-1}w+xw^{d-1}+y^{d-1}z+yz^{d-1}).\
\]
}}
\fi

Since we have already seen that $d-1$ is a power of $p$ for {\it some} prime $p$ (\cref{has num star chord2}), the surface $S = \mathbb{V}(f)$ would be extremal\footnote{see \cref{extremaldef} and the discussion thereafter} and the proof would be complete {\it  if that prime $p$ is equal to the characteristic of $k$.} This remains to be shown.

To this end, we consider the lines of $S=\mathbb V(x^{d-1}w+xw^{d-1}+y^{d-1}z+yz^{d-1})$ properly intersecting the line $L = \mathbb{V}(z,w)$. We know that there are $(d-1)(d^2-2d+2)$ lying on $S$ and intersecting $L$.
Any  line $M$ intersecting $L$ has defining equations
\begin{equation}\label{lines-7}
\begin{aligned}
a_1x-a_2y-a_3z - a_4w &=0\\
b_3z - b_4w&=0\\
\end{aligned}
\end{equation}
since these linear forms, together with $\{z, w\}$, must be linearly dependent. The case where $b_3=0$ produces the $d$ lines of the plane section $S\cap \mathbb V(w)$, including $L$ itself, and the case where  $b_4=0$ produces the $d$ lines of the plane section $S\cap \mathbb V(z)$, including $L$ itself. 
Thus, there are $(d-1)(d^2-2d+2)-2(d-1)= d(d-1)(d-2)$ {\it remaining} lines  $M$ on $S$ of the form 
\begin{equation}\label{lines-8}
\begin{aligned}
a_1x-a_2y\,\,\,\,\,\,\,\,\,\,\,\, - a_4w &=0\\
b_3z - b_4w&=0\\
\end{aligned}
\end{equation}
where $b_3,b_4\neq 0$
on $S$.
Assuming $M$ is in neither plane $\mathbb V(w)$ nor $\mathbb V(z)$, note that neither $a_1$ nor $a_2$ is zero:  if $a_1=0$, the lines  $M$, $ \mathbb V(z, w)$, and $ \mathbb V(y, w)$ all pass through the point $[1:0:0:0]$,
 forcing them to be coplanar (\cref{threelines}), so that $M$ would be in the plane $\mathbb V(w)$. Likewise, if 
 $a_2= 0$, the point $[0:1:0:0]$ lies on the lines $M$, $\mathbb V(x, z)$, and $\mathbb V(w, z)$ forcing $M$ to be in the plane $\mathbb V(z)$. 
  Without loss of generality, we assume $a_1=b_3=1$. 

Now, consider the restriction of forms to the line $M$, which we view as
the map 
\[
k[x, y, z, w] \rightarrow k[y, w] \,\,\,\,\,\, \,\,\,\,\,\ x\mapsto a_2y+a_4w,\,\,\,z \mapsto b_4w.
\]
The polynomial $f$ \eqref{almost} restricts to 
\begin{equation}\label{f-bar2}
\bar{f}  = (a_2y+a_4w)^{d-1}w+(a_2y+a_4w)w^{d-1}+y^{d-1}(b_4w)+y(b_4w)^{d-1},
\end{equation}
which is the zero polynomial if and only if the line \eqref{lines-8} lies on $S$.
Because there are $d(d-1)(d-2)$ lines $M$ on $S$ of the form \eqref{lines-8}, there must be $d(d-1)(d-2)$ triples $(a_2, a_4, b_4)$ where $\bar f$ is the zero polynomial.
But when $d>3$,  the coefficient of $y^{d-2}w^2$  in \eqref{f-bar2} is the polynomial 
 \[
(d-1) a_2^{d-2}a_4.
\]
Since $a_2\neq 0$, either $a_4$ or $d-1$ is zero in the field $k$.
If $a_4=0$, then $\bar f = (a_2^{d-1}+b_4) y^{d-1}w + (a_2+b_4^{d-1}) yw^{d-1}$. 
Setting those two coefficients to zero, we see there are at most $(d-1)^2$ such lines.
Since there are $d(d-1)(d-2)$ lines on $S$ of the form \eqref{lines-8}, there must be some such lines with $a_4\neq0$, in which case $(d-1)$ must be zero in the field $k$.
In other words, the characteristic of $k$ must divide $d-1=p^e$. We conclude that $k$ has characteristic $p$,  completing the proof. 
\end{proof}
\bibliographystyle{amsalpha}
\bibliography{Citations}

\newcommand{\etalchar}[1]{$^{#1}$}
\providecommand{\bysame}{\leavevmode\hbox to3em{\hrulefill}\thinspace}
\providecommand{\MR}{\relax\ifhmode\unskip\space\fi MR }
\providecommand{\MRhref}[2]{%
  \href{http://www.ams.org/mathscinet-getitem?mr=#1}{#2}
}
\providecommand{\href}[2]{#2}
\begin{thebibliography}{KKP{\etalchar{+}}22}

\bibitem[BC66]{Bose-Chakravarti}
R.~C. Bose and I.~M. Chakravarti, \emph{Hermitian varieties in a finite
  projective space {${\rm PG}(N,\,q^{2})$}}, Canadian J. Math. \textbf{18}
  (1966), 1161--1182. \MR{200782}

\bibitem[Bea90]{Beauville.90}
A.~Beauville, \emph{Sur les hypersurfaces dont les sections hyperplanes sont
  \`a module constant}, The {G}rothendieck {F}estschrift, {V}ol. {I}, Progr.
  Math., vol.~86, Birkh\"{a}user Boston, Boston, MA, 1990, With an appendix by
  David Eisenbud and Craig Huneke, pp.~121--133. \MR{1086884}

\bibitem[BPRS24]{Brosowsky+Page+Ryan+Smith}
Anna Brosowsky, Janet Page, Tim Ryan, and Karen~E. Smith, \emph{Geometry of
  smooth extremal surfaces}, J. Algebra \textbf{646} (2024), 376--411.
  \MR{4711041}

\bibitem[BR23]{Bauer+Rams}
Thomas Bauer and S{\l}awomir Rams, \emph{Counting lines on projective
  surfaces}, Annali della Scuola Normale Superiore di Pisa. Classe di Scienze.
  Serie V \textbf{24} (2023), no.~3, 1285--1299.

\bibitem[BS07]{Boissiere+Sarti}
Samuel Boissi\`ere and Alessandra Sarti, \emph{Counting lines on surfaces},
  Ann. Sc. Norm. Super. Pisa Cl. Sci. (5) \textbf{6} (2007), no.~1, 39--52.
  \MR{2341513}

\bibitem[CC10]{Cools-Coppens}
Filip Cools and Marc Coppens, \emph{Star points on smooth hypersurfaces}, J.
  Algebra \textbf{323} (2010), no.~1, 261--286. \MR{2564838}

\bibitem[Che23a]{Che23a}
Raymond Cheng, \emph{{q}-bic forms}, preprint (2023).

\bibitem[Che23b]{Che23b}
\bysame, \emph{{q}-bic hypersurfaces and their {F}ano schemes}, preprint
  (2023).

\bibitem[CHM95]{Caporaso+Harris+Mazur}
Lucia Caporaso, Joe Harris, and Barry Mazur, \emph{How many rational points can
  a curve have?}, The moduli space of curves ({T}exel {I}sland, 1994), Progr.
  Math., vol. 129, Birkh\"auser Boston, Boston, MA, 1995, pp.~13--31.
  \MR{1363052}

\bibitem[Cle61]{Clebsch}
A.~Clebsch, \emph{Zur {T}heorie der algebraischen {F}l\"achen}, J. Reine Angew.
  Math. \textbf{58} (1861), 93--108. \MR{1579142}

\bibitem[Deg15]{Degtyarev.15}
Alex Degtyarev, \emph{Lines generate the {P}icard groups of certain {F}ermat
  surfaces}, J. Number Theory \textbf{147} (2015), 454--477. \MR{3276333}

\bibitem[Deg22]{Degtyarev22}
\bysame, \emph{Lines in supersingular quartics}, J. Math. Soc. Japan
  \textbf{74} (2022), no.~3, 973--1019. \MR{4484237}

\bibitem[DIS17]{Degtyarev+Itenberg+Sertoz}
Ali~Sinan Degtyarev-Itenberg-Sertoz, \emph{Lines on quartic surfaces}, Math.
  Ann. \textbf{368} (2017), no.~1-2, 753--809. \MR{3651588}

\bibitem[DL76]{Deligne+Lusztig}
P.~Deligne and G.~Lusztig, \emph{Representations of reductive groups over
  finite fields}, Ann. of Math. (2) \textbf{103} (1976), no.~1, 103--161.
  \MR{393266}

\bibitem[DR23]{Degtyarev+Rams}
Alex Degtyarev and Sławomir Rams, \emph{Lines on k3-quartics via triangular
  sets}, 2023.

\bibitem[Ful98]{Fulton}
William Fulton, \emph{Intersection theory}, second ed., Ergebnisse der
  Mathematik und ihrer Grenzgebiete. 3. Folge. A Series of Modern Surveys in
  Mathematics [Results in Mathematics and Related Areas. 3rd Series. A Series
  of Modern Surveys in Mathematics], vol.~2, Springer-Verlag, Berlin, 1998.
  \MR{1644323}

\bibitem[GAR16]{Gonzalez+Rams}
V\'{\i}ctor Gonz\'{a}lez-Alonso and S{\l}awomir Rams, \emph{Counting lines on
  quartic surfaces}, Taiwanese J. Math. \textbf{20} (2016), no.~4, 769--785.
  \MR{3535673}

\bibitem[Har95]{Harris}
Joe Harris, \emph{Algebraic geometry}, Graduate Texts in Mathematics, vol. 133,
  Springer-Verlag, New York, 1995, A first course, Corrected reprint of the
  1992 original. \MR{1416564}

\bibitem[Hef85]{Hef85}
Abramo Hefez, \emph{D{UALITY} {FOR} {PROJECTIVE} {VARIETIES}}, ProQuest LLC,
  Ann Arbor, MI, 1985, Thesis (Ph.D.)--Massachusetts Institute of Technology.
  \MR{2941091}

\bibitem[Hir85]{Hirschfeld}
J.~W.~P. Hirschfeld, \emph{Finite projective spaces of three dimensions},
  Oxford Mathematical Monographs, The Clarendon Press, Oxford University Press,
  New York, 1985, Oxford Science Publications. \MR{840877}

\bibitem[HK13]{Homma-Kim.13}
Masaaki Homma and Seon~Jeong Kim, \emph{An elementary bound for the number of
  points of a hypersurface over a finite field}, Finite Fields Appl.
  \textbf{20} (2013), 76--83. \MR{3015353}

\bibitem[HK15]{Homma-Kim.15}
\bysame, \emph{Numbers of points of surfaces in the projective 3-space over
  finite fields}, Finite Fields Appl. \textbf{35} (2015), 52--60. \MR{3368799}

\bibitem[HK16]{Homma-Kim.16}
\bysame, \emph{The characterization of {H}ermitian surfaces by the number of
  points}, J. Geom. \textbf{107} (2016), no.~3, 509--521. \MR{3563207}

\bibitem[HW36]{Hasse-Witt}
Helmut Hasse and Ernst Witt, \emph{Zyklische unverzweigte
  {E}rweiterungsk\"orper vom {P}rimzahlgrade {$p$} \"uber einem algebraischen
  {F}unktionenk\"orper der {C}harakteristik {$p$}}, Monatsh. Math. Phys.
  \textbf{43} (1936), no.~1, 477--492. \MR{1550551}

\bibitem[Igu60]{Igusa}
Jun-ichi Igusa, \emph{Betti and {P}icard numbers of abstract algebraic
  surfaces}, Proc. Nat. Acad. Sci. U.S.A. \textbf{46} (1960), 724--726.
  \MR{118725}

\bibitem[KKP{\etalchar{+}}22]{LowerBoundsExtremal}
Zhibek Kadyrsizova, Jennifer Kenkel, Janet Page, Jyoti Singh, Karen~E. Smith,
  Adela Vraciu, and Emily~E. Witt, \emph{Lower bounds on the {$F$}-pure
  threshold and extremal singularities}, Trans. Amer. Math. Soc. Ser. B
  \textbf{9} (2022), 977--1005. \MR{4498775}

\bibitem[KLOS23]{Kollar+Lieblich+Olsson+Sawin}
J\'{a}nos Koll\'{a}r, Max Lieblich, Martin Olsson, and Will Sawin, \emph{What
  determines an algebraic variety?}, Annals of Mathematics Studies, vol. 216,
  Princeton University Press, Princeton, NJ, [2023] \copyright 2023.
  \MR{4648834}

\bibitem[Kol15]{Kollar}
J\'{a}nos Koll\'{a}r, \emph{Szemer\'{e}di-{T}rotter-type theorems in dimension
  3}, Adv. Math. \textbf{271} (2015), 30--61. \MR{3291856}

\bibitem[KP91]{Kleimann+Piene}
Steven Kleiman and Ragni Piene, \emph{On the inseparability of the {G}auss
  map}, Enumerative algebraic geometry ({C}openhagen, 1989), Contemp. Math.,
  vol. 123, Amer. Math. Soc., Providence, RI, 1991, pp.~107--129. \MR{1143550}

\bibitem[Lan56]{Lan56}
S.~Lang, \emph{Algebraic groups over finite fields}, Amer. J. Math. \textbf{78}
  (1956), 555--563. \MR{86367}

\bibitem[Mat57]{Matsusaka}
T.~Matsusaka, \emph{The criteria for algebraic equivalence and the torsion
  group}, Amer. J. Math. \textbf{79} (1957), 53--66. \MR{82730}

\bibitem[Miy09]{Miyaoka}
Yoichi Miyaoka, \emph{Counting lines and conics on a surface}, Publ. Res. Inst.
  Math. Sci. \textbf{45} (2009), no.~3, 919--923. \MR{2569571}

\bibitem[Nom95]{Noma}
Atsushi Noma, \emph{Hypersurfaces with smooth dual varieties}, Amer. J. Math.
  \textbf{117} (1995), no.~6, 1507--1515. \MR{1363077}

\bibitem[PT09]{PayneThasBook}
Stanley~E. Payne and Joseph~A. Thas, \emph{Finite generalized quadrangles},
  second ed., EMS Series of Lectures in Mathematics, European Mathematical
  Society (EMS), Z\"{u}rich, 2009. \MR{2508121}

\bibitem[RS14]{Rams+Schutt.14}
S{\l}awomir Rams and Matthias Sch\"utt, \emph{On quartics with lines of the
  second kind}, Adv. Geom. \textbf{14} (2014), no.~4, 735--756. \MR{3276131}

\bibitem[RS15a]{Rams+Schutt.15-112lines}
S{\l}awomir Rams and Matthias Sch\"{u}tt, \emph{112 lines on smooth quartic
  surfaces (characteristic 3)}, Q. J. Math. \textbf{66} (2015), no.~3,
  941--951. \MR{3396099}

\bibitem[RS15b]{Rams+Schutt.15-64lines}
\bysame, \emph{64 lines on smooth quartic surfaces}, Math. Ann. \textbf{362}
  (2015), no.~1-2, 679--698. \MR{3343894}

\bibitem[RS18]{Rams+Schutt.18-AtMost64linesChar2}
\bysame, \emph{At most 64 lines on smooth quartic surfaces (characteristic 2)},
  Nagoya Math. J. \textbf{232} (2018), 76--95. \MR{3866501}

\bibitem[RS20]{Rams+Schutt.20-Quintics}
S{\l}awomir Rams and Matthias Sch\"utt, \emph{Counting lines on surfaces,
  especially quintics}, Ann. Sc. Norm. Super. Pisa Cl. Sci. (5) \textbf{20}
  (2020), no.~3, 859--890. \MR{4166795}

\bibitem[Sch82]{Schur}
Friedrich Schur, \emph{Ueber eine besondre {C}lasse von {F}l\"achen vierter
  {O}rdnung}, Math. Ann. \textbf{20} (1882), no.~2, 254--296. \MR{1510168}

\bibitem[Seg43a]{Segre43}
B.~Segre, \emph{The maximum number of lines lying on a quartic surface}, Quart.
  J. Math. Oxford Ser. \textbf{14} (1943), 86--96. \MR{10431}

\bibitem[Seg43b]{Segre.43}
\bysame, \emph{The maximum number of lines lying on a quartic surface}, Quart.
  J. Math. Oxford Ser. \textbf{14} (1943), 86--96. \MR{10431}

\bibitem[Seg65]{Segre.65}
Beniamino Segre, \emph{Forme e geometrie hermitiane, con particolare riguardo
  al caso finito}, Ann. Mat. Pura Appl. (4) \textbf{70} (1965), 1--201.
  \MR{213949}

\bibitem[Seg67]{Segre}
\bysame, \emph{Introduction to {G}alois geometries}, Atti Accad. Naz. Lincei
  Mem. Cl. Sci. Fis. Mat. Natur. Sez. Ia (8) \textbf{8} (1967), 133--236.
  \MR{238846}

\bibitem[SGA73]{SGAII}
\emph{Groupes de monodromie en g\'{e}om\'{e}trie alg\'{e}brique. {II}}, Lecture
  Notes in Mathematics, vol. Vol. 340, Springer-Verlag, Berlin-New York, 1973,
  S\'{e}minaire de G\'{e}om\'{e}trie Alg\'{e}brique du Bois-Marie 1967--1969
  (SGA 7 II), Dirig\'{e} par P. Deligne et N. Katz. \MR{354657}

\bibitem[Shi01]{Shi01}
Ichiro Shimada, \emph{Lattices of algebraic cycles on {F}ermat varieties in
  positive characteristics}, Proc. London Math. Soc. (3) \textbf{82} (2001),
  no.~1, 131--172. \MR{1794260}

\bibitem[SK79]{Shioda+Katsura}
Tetsuji Shioda and Toshiyuki Katsura, \emph{On {F}ermat varieties}, Tohoku
  Math. J. (2) \textbf{31} (1979), no.~1, 97--115. \MR{526513}

\bibitem[SS17]{Shimada+Shioda}
Ichiro Shimada and Tetsuji Shioda, \emph{On a smooth quartic surface containing
  56 lines which is isomorphic as a {$K3$} surface to the {F}ermat quartic},
  Manuscripta Math. \textbf{153} (2017), no.~1-2, 279--297. \MR{3635983}

\bibitem[SSvL10]{Schutt+Shioda+vanLuijk}
Matthias Sch\"utt, Tetsuji Shioda, and Ronald van Luijk, \emph{Lines on
  {F}ermat surfaces}, J. Number Theory \textbf{130} (2010), no.~9, 1939--1963.
  \MR{2653207}

\bibitem[Tha78]{Thas78}
J.~A. Thas, \emph{Combinatorial characterizations of generalized quadrangles
  with parameters {$s=q$} and {$t=q\sp{2}$}}, Geom. Dedicata \textbf{7} (1978),
  no.~2, 223--232. \MR{485452}

\bibitem[Tha94]{Thas-94}
\bysame, \emph{Generalized quadrangles of order {$(s,s^2)$}. {I}}, J. Combin.
  Theory Ser. A \textbf{67} (1994), no.~2, 140--160. \MR{1284404}

\bibitem[Tit59]{Tits}
Jacques Tits, \emph{Sur la trialit\'{e} et certains groupes qui s'en
  d\'{e}duisent}, Inst. Hautes \'{E}tudes Sci. Publ. Math. (1959), no.~2,
  13--60. \MR{1557095}

\bibitem[Ven17a]{Veniani-Char2}
Davide~Cesare Veniani, \emph{Lines on {K}3 quartic surfaces in characteristic
  2}, Q. J. Math. \textbf{68} (2017), no.~2, 551--581. \MR{3667213}

\bibitem[Ven17b]{Veniani-K3Quartic}
\bysame, \emph{The maximum number of lines lying on a {K}3 quartic surface},
  Math. Z. \textbf{285} (2017), no.~3-4, 1141--1166. \MR{3623744}

\bibitem[Ven20]{Veniani-QuarticSymmetries}
\bysame, \emph{Symmetries and equations of smooth quartic surfaces with many
  lines}, Rev. Mat. Iberoam. \textbf{36} (2020), no.~1, 233--256. \MR{4061988}

\bibitem[Ven22]{Veniani-Char3}
\bysame, \emph{Lines on {K}3 quartic surfaces in characteristic 3}, Manuscripta
  Math. \textbf{167} (2022), no.~3-4, 675--701. \MR{4385387}

\bibitem[Xia19]{Xiao}
Stanley~Yao Xiao, \emph{On binary cubic and quartic forms}, J. Th\'eor. Nombres
  Bordeaux \textbf{31} (2019), no.~2, 323--341. \MR{4030910}

\end{thebibliography}
\end{document}